\documentclass[10pt,reqno]{amsart}

\usepackage{amsthm, mathrsfs, amsmath, amstext, amsxtra, amsfonts, dsfont, amssymb, bm}
\usepackage{xcolor}
\usepackage{lmodern}
\usepackage{relsize}
\usepackage[T1]{fontenc}
\usepackage[colorlinks, linkcolor=red, citecolor=blue, urlcolor=blue, 
hypertexnames=false]{hyperref}
\usepackage{tikz,tkz-tab}

\usepackage{geometry}
\newcommand{\R}{\mathbb R}

\newcommand{\Ac}{\mathcal A}

\newcommand{\vareps}{\varepsilon}

\DeclareMathOperator*{\ST}{ST}

\DeclareMathOperator*{\loc}{loc}
\DeclareMathOperator*{\rad}{rad}

\newcommand{\scal}[1]{\left\langle #1 \right\rangle} 

\newtheorem{theorem}{Theorem}[section]
\newtheorem{lemma}[theorem]{Lemma} 
\newtheorem{proposition}[theorem]{Proposition}
\newtheorem{corollary}[theorem]{Corollary} 
\theoremstyle{definition}
\newtheorem{definition}[theorem]{Definition}
\newtheorem{remark}[theorem]{Remark}

\title[Focusing Exponential NLS]{Long time dynamics for the focusing nonlinear Schr\"odinger equation with exponential nonlinearities} 

\author[V. D. Dinh, S. Keraani, and M. Majdoub]{Van Duong Dinh, Sahbi Keraani, and Mohamed Majdoub}

\address[V. D. Dinh]{Laboratoire Paul Painlev\'e UMR 8524, Universit\'e de Lille CNRS, 59655 Villeneuve d'Ascq Cedex, France
	\&
	Department of Mathematics, HCMC University of Pedagogy, 280 An Duong Vuong, Ho Chi Minh, Vietnam}
\email{\sl contact@duongdinh.com}

\address[S. Keraani]{Laboratoire Paul Painlev\'e UMR 8524, Universit\'e de Lille CNRS, 59655 Villeneuve d'Ascq Cedex, France}
\email{\sl sahbi.keraani@univ-lille.fr}

\address[M. Majdoub]{Deapartment of Mathematics, College of Science, Imam Abdulrahman Bin Faisal University, P. O. Box 1982, Dammam, Saudi Arabia
	\&
	Basic and Applied Scientific Research Center, Imam Abdulrahman Bin Faisal University, P.O. Box 1982, 31441, Dammam, Saudi Arabia}
\email{\sl mmajdoub@iau.edu.sa}

\subjclass[2010]{35Q55, 35Q44; 35P25}
\keywords{Nonlinear Schr\"odinger equation, Exponential nonlinearity, Ground state, Scattering, Blow-up}

\begin{document}
	
	\begin{abstract}
		In this paper, we study the focusing nonlinear Schr\"odinger equation with exponential nonlinearities
		\[
		\left\{
		\begin{array}{rcl}
		i \partial_t u + \Delta u &=& - \left(e^{4\pi |u|^2} - 1 - 4\pi \mu |u|^2 \right) u, \quad (t,x) \in \R \times \R^2, \\
		u(0) &=& u_0 \in H^1,
		\end{array}
		\right.
		\]
		where $\mu \in \{0, 1\}$. By using variational arguments, we derive invariant sets where the global existence and finite time blow-up occur. In particular, we obtain sharp thresholds for global existence and finite time blow-up. In the case $\mu=1$, we show the asymptotic behavior or energy scattering of global solutions by using a recent argument of Arora-Dodson-Murphy \cite{ADM}. 
	\end{abstract}
	
	\maketitle

	\section{Introduction}
	\label{S1}
	\setcounter{equation}{0}
	
	We consider the initial valued problem for nonlinear Schr\"odinger equations with exponential nonlinearities
	\begin{align} \label{NLS}
	\left\{
	\begin{array}{ccl}
	i\partial_t u + \Delta u &=& - f_\mu(u), \quad (t,x) \in \R \times \R^2,  \\
	u(0,x)  &= & u_0(x),
	\end{array}
	\right.
	\end{align}
	where 
	\begin{align} \label{defi-f-mu}
	f_\mu(u) = \left(e^{4\pi |u|^2} - 1 - 4\pi \mu |u|^2 \right) u, \quad \mu \in \{0,1\}.
	\end{align}
	The nonlinear Schr\"odinger equation (NLS) with exponential nonlinearity arises in several physical contexts such as the self-trapped beams in plasma (see e.g. \cite{LLT}). To our knowledge, the first paper studied NLS with exponential nonlinearity goes back to Cazenave \cite{Cazenave-expo} where he considered the Schr\"odinger equation with $f(u) = \left(1-e^{-|u|^2}\right) u$ and showed the global well-posedness and scattering. In this setting, the function $s \mapsto f(s)$ is uniformly bounded together with all its derivatives due to the negative exponent. In our setting, the nonlinearities and their derivatives grow more rapidly than any power for large amplitude. This makes our problem more difficult comparing to the one in \cite{Cazenave-expo}. Another interest of considering \eqref{defi-f-mu} is their relations to the Trudinger-Moser inequality (see Section $\ref{S2}$).
	
	Solutions to \eqref{NLS} formally enjoy the conservation of mass and energy, namely
	\begin{align*}
	M(u(t)) &= \|u(t)\|^2_{L^2} = M(u_0), \tag{Mass} \\
	E_\mu(u(t)) &= \frac{1}{2} \|\nabla u(t)\|^2_{L^2} - \int F_\mu(u(t)) dx = E_\mu(u_0), \tag{Energy} \\
	\end{align*}
	where
	\[
	F_\mu(u) := \frac{1}{8\pi} \left(e^{4\pi |u|^2} - 1 - 4\pi |u|^2 - 8\pi^2 \mu |u|^4 \right).
	\]
	
	The local well-posedness for \eqref{NLS} has been established by Colliander-Ibrahim-Majdoub-Masmoudi \cite{CIMM}. More precisely, the following result holds.
	\begin{theorem} [\cite{CIMM}] \label{theo-lwp}
		Let $u_0 \in H^1$ be such that $\|\nabla u_0\|_{L^2} <1$. Then there exists $T>0$ and a unique solution $u$ to \eqref{NLS} in $C([0,T], H^1)$. Moreover, $u \in L^4([0,T], \mathcal{C}^{1/2})$ and for all $0\leq t\leq T$, $E_\mu(u(t))=E_\mu(u_0)$ and $M(u(t))=M(u_0)$. Here $\mathcal{C}^\alpha$ denotes the space of $\alpha$-H\"older continuous functions endowed with the norm
		\[
		\|u\|_{\mathcal{C}^\alpha} := \|u\|_{L^\infty} + \sup_{x \ne y} \frac{|u(x) - u(y)|}{|x-y|^\alpha}.
		\]
	\end{theorem}
	Let $T^*$ is the maximal forward time of existence, i.e.
	\[
	T^*:= \sup\{ T>0 \ : \ \text{there exists a solution to \eqref{NLS} on } [0,T]\}.
	\]
	We have the blow-up alternative: either $T^*=+\infty$ or $T^*<+\infty$ and
	\begin{align} \label{blow-alt}
	\limsup_{t\nearrow T^*} \|\nabla u(t)\|_{L^2} =1.
	\end{align}
	

	The main purpose of this paper is to study long time dynamics such as global existence, blow-up and energy scattering for the equation \eqref{NLS}. Before stating our results, let us recall some known results related to \eqref{NLS}. In the defocusing case, i.e. the plus sign in front of the nonlinearity, the global well-posedness in $H^1$ was investigated by Colliander-Ibrahim-Majdoub-Masmoudi \cite{CIMM}. They introduced the notion of criticality as follows: the defocusing problem \eqref{NLS} is said to be subcritical if the energy is strictly smaller than $\frac{1}{2}$, critical if the energy is equal to $\frac{1}{2}$ and supercritical if the energy is strictly greater than $\frac{1}{2}$. They proved that the equation is globally well-posed in $H^1$ in both subcritical and critical regimes, and global solutions satisfy $u \in C(\R, H^1) \cap L^4_{\loc}(\R, W^{1,4})$. Moroever, a sort of ill-posedness was proved in the supercritical case. More precisely, the solution maps $u_0 \mapsto u(t)$ fails to be continuous in $H^1$ as $t\rightarrow 0$.	Afterwards, the energy scattering for the defocusing problem \eqref{NLS} with $\mu=1$ in the subcritical case was established by Ibrahim-Majdoub-Masmoudi-Nakanishi \cite{IMMN}. The proof is based on the a priori global bound $\|u\|_{L^4(\R, L^8)} \leq C(M, E)<\infty$ which was proved independently by Colliander-Grillakis-Tzirakis \cite{CGT} and Planchon-Vega \cite{PV}. Later, the energy scattering with radially symmetric initial data for the defocusing problem \eqref{NLS} with $\mu=1$ in the critical case was proved by Bahouri-Ibrahim-Perelman \cite{BIP}. The proof relies on both the a priori global bound $\|u\|_{L^4(\R,L^8)}$ and the characterization of the lack of compactness of the Sobolev embedding $H^1_{\rad}$ into the critical Orlicz space \cite{BMM}. Recently, Azzam \cite{AZZ} proved the energy scattering for the defocusing problem \eqref{NLS} with $\mu=0$ in the subcritical case. The proof is based on the perturbative argument of \cite{TVZ} by viewing the nonlinearity $f_0$ as a perturbation of the mass-critical NLS. This allows the author to combine the a priori global bound $\|u\|_{L^4(\R, L^8)}$ and the known spacetime estimate for the mass-critical NLS proved by Dodson \cite{Dodson} to obtain the global bound $\|u\|_{L^4(\R, W^{1,4})}$.  
	
	To state our results, let us recall the following notion of ground states related to \eqref{NLS}. By standing wave solutions, we mean solutions to \eqref{NLS} of the form $u(t,x) = e^{it} \phi(x)$, where $\phi \in H^1$ solves the elliptic equation
	\begin{align} \label{ell-equ}
	-\Delta \phi + \phi = f_\mu(\phi).
	\end{align}
	\begin{definition} [Ground state] 
		A non-zero $H^1$ solution $Q$ to \eqref{ell-equ} is called a ground state related to \eqref{ell-equ} if it minimizes the action functional
		\[
		S_\mu(\phi):= E_\mu(\phi) + \frac{1}{2} M(\phi) = \frac{1}{2} \|\nabla \phi\|^2_{L^2} + \frac{1}{2} \|\phi\|^2_{L^2} - \int F_\mu(\phi) dx
		\]
		over all non-trivial solution of \eqref{ell-equ}, that is,
		\[
		S_\mu(Q) = \inf \left\{ S_\mu(\phi) \ : \ \phi \in H^1 \backslash \{0\}, \phi \text{ is a solution to } \eqref{ell-equ} \right\}.
		\]
	\end{definition}
	The existence of ground states related to \eqref{ell-equ} has been studied by many authors. In \cite{JT}, Jeanjean-Tanaka proved a mountain pass characterization of ground states related to \eqref{ell-equ} when the nonlinearity has a subcritical exponential growth. Alves-Souto-Montenegro \cite{ASM} improved the arguments of \cite{JT} by assuming the nonlinearity has a critical exponential growth. Recently, Ruf-Sani \cite{RS} extended Montenegro-Souto's results to a more general class of critical exponetial nonlinearities. More precisely, they proved the following result.
	\begin{theorem}[Existence of ground states \cite{RS}] \label{theo-GS}
		Let $f$ satisfy the following conditions:
		\begin{itemize}
			\item[i.] $f:\R \rightarrow \R$ is continuous and has critical exponential growth, i.e.
			\[
			\lim_{|s| \rightarrow \infty} \frac{|f(s)|}{e^{\alpha s^2}} = \left\{
			\begin{array}{cl}
			0 &\text{if } \alpha>4\pi, \\
			+\infty &\text{if } \alpha<4\pi.
			\end{array}
			\right.
			\]
			\item[ii.] $\displaystyle\lim_{s \rightarrow 0} \frac{f(s)}{s} =0$.
			\item[iii.] There exists $\delta>2$ such that $0<\delta F(s) <sf(s)$ for any $s \ne 0$, where $F(s):=\mathlarger{\int}_0^s f(\tau)d\tau$.
			\item[iv.] $\displaystyle\lim_{|s| \rightarrow +\infty} \frac{sf(s)}{e^{4\pi s^2}} >0$. 
		\end{itemize}
		Then there exists a ground state $Q$ related to 
		\begin{align} \label{ell-equ-f}
		-\Delta \phi + \phi = f(\phi)
		\end{align}
		which is radially symmetric. In addition,
		\[
		\frac{1}{2} \|\nabla Q\|^2_{L^2} = \inf \left\{\frac{1}{2} \|\nabla \phi\|^2_{L^2} \ : \ \phi \in H^1 \backslash \{0\}, \frac{1}{2} \|\phi\|^2_{L^2} = \int F(\phi) dx \right\}.
		\]
		Moreover, 
		\[
		0<\|\nabla Q\|_{L^2} <1.
		\]
	\end{theorem}
	
	We collect some properties of the ground state $Q$. 
	\begin{lemma} \label{lem-pro-Q}
		The ground state $Q$ obtained in Theorem $\ref{theo-GS}$ satisfies the following properties:
		\begin{itemize}
			\item $Q \in C^2 \cap L^\infty$ and $Q$ decays exponentially at infinity.
			\item $Q$ is radially symmetric.
			\item $0<\|\nabla Q\|_{L^2} <1$.
			\item $\|\nabla Q\|^2_{L^2} + \|Q\|^2_{L^2} = \mathlarger{\int} \overline{Q} f(Q) dx$.
			\item $\frac{1}{2} \|Q\|^2_{L^2} = \mathlarger{\int} F(Q) dx$.
		\end{itemize}
	\end{lemma}
	\begin{proof}
		The first item follows from \cite[Proposition 2.1]{ZO}. The second and third items follow from Theorem $\ref{theo-GS}$. Multiplying both sides of \eqref{ell-equ-f} with $\overline{Q}$, then integrating over $\R^2$ and performing integration by parts, we get the fourth item. The last item follows by multiplying \eqref{ell-equ-f} with $x\cdot \nabla \overline{Q}$ and integrating over $\R^2$. 
	\end{proof}
	It is easy to check that our nonlinearities $f_\mu$ (see \eqref{defi-f-mu}) satisfy the assumptions i--iv of Theorem $\ref{theo-GS}$. Thus, there exist ground states $Q_\mu$ related to \eqref{ell-equ} which satisfy the properties given in Lemma $\ref{lem-pro-Q}$. It follows that
	\begin{align} \label{chara-P}
	S_\mu(Q_\mu)=\frac{1}{2}\|\nabla Q_\mu\|^2_{L^2} = \inf \left\{S_\mu(\phi) \ : \ \phi \in H^1 \backslash \{0\}, P_\mu(\phi)=0\right\},
	\end{align}
	where
	\[
	P_\mu(\phi):= \frac{1}{2} \|\phi\|^2_{L^2} - \int F_\mu(\phi) dx.
	\]
	Note that if $P_\mu(\phi)=0$, then $S_\mu(\phi) = \frac{1}{2} \|\nabla \phi\|^2_{L^2}$. Let us define the following sets
	\begin{align} \label{inv-sets}
	\begin{aligned}
	\mathcal{A}_\mu^+ &:= \left\{\phi \in H^1 \backslash \{0\} \ : \ S_\mu(\phi)<S_\mu(Q_\mu), P_\mu(\phi)>0 \right\}, \\
	\mathcal{A}_\mu^- &:= \left\{\phi \in H^1 \backslash \{0\} \ : \ S_\mu(\phi)<S_\mu(Q_\mu), P_\mu(\phi)<0 \right\}.
	\end{aligned}
	\end{align}
	Note that by \eqref{chara-P}, 
	\begin{align} \label{union-A}
	\mathcal{A}_\mu^+ \cup \mathcal{A}_\mu^- = \left\{ \phi \in H^1 \backslash \{0\} \ : \ S_\mu(\phi)<S_\mu(Q_\mu) \right\}
	\end{align}
	since
	\[
	\left\{\phi \in H^1 \backslash \{0\} \ : \ S_\mu(\phi)<S_\mu(Q_\mu), P_\mu(\phi)=0 \right\} = \emptyset.
	\]
	By the continuity argument and \eqref{chara-P}, it is easy to see that the sets $\mathcal{A}_\mu^\pm$ are invariant under the flow of \eqref{NLS}.
	
	Our first result is the following global existence for \eqref{NLS}.
	\begin{theorem} [Global existence] \label{theo-gwp}
		Let $\mu \in \{0,1\}$ and $u_0 \in \mathcal{A}_\mu^+$. Then the corresponding solution to \eqref{NLS} exists globally in time.
	\end{theorem}
	
	Our next result concerns the finite time blow-up for \eqref{NLS}. 
	
	\begin{theorem} [Finite time blow-up] \label{theo-blowup}
		Let $\mu \in \{0,1\}$. Let $u_0 \in H^1$ be such that $\|\nabla u_0\|_{L^2}<1$.
		\begin{itemize}
		\item  If $E_\mu(u_0)<0$ and either $u_0 \in L^2(|x|^2 dx)$ or $u_0$ is radially symmetric, then the corresponding solution to \eqref{NLS} blows up in finite time.
		\item If $E_\mu(u_0) \geq 0$, $u_0 \in \mathcal{A}^-_\mu$ and either $u_0 \in L^2(|x|^2 dx)$ or $u_0$ is radially symmetric, then the corresponding solution to \eqref{NLS} blows up in finite time.
		\end{itemize} 
	\end{theorem}
	
	The proof of the finite time blow-up is closely related to the virial functional
	\[
	I_\mu(\phi):= \|\nabla \phi\|_{L^2}^2 - \int \overline{\phi} f_\mu(\phi) - 2 F_\mu(\phi) dx = 2 E_\mu(\phi) - \int \overline{\phi} f_\mu(\phi) - 4 F_\mu(\phi) dx.
	\]
	The functional $I_\mu$ is nothing but the second time derivative of $\|xu(t)\|^2_{L^2}$ (see \eqref{virial-proof}), namely
	\begin{align} \label{virial}
	\frac{d^2}{dt^2} \|xu(t)\|^2_{L^2} = 8 I_\mu(u(t)), \quad \forall t\in [0,T^*).
	\end{align}
	The finite time blow-up for negative energy initial data follows easily by noting that
	\[
	\int \overline{\phi} f_\mu(\phi) - 4 F_\mu(\phi) dx\geq 0, \quad \forall \phi \in H^1.
	\]
	The one for non-negative energy initial data is more involved. To this end, we observe (see \eqref{S-I}) 
	\begin{align} \label{S-I-intro}
	S_\mu(Q_\mu) = \inf \left\{ S_\mu(\phi) \ : \ \phi \in H^1 \backslash \{0\}, I_\mu(\phi) =0 \right\}
	\end{align}
	and define 
	\begin{align} \label{def-K-pm}
	\mathcal{K}_\mu^- &:= \left\{ \phi \in H^1 \backslash \{0\} \ : \ S_\mu(\phi)<S_\mu(Q_\mu), I_\mu(\phi)<0\right\}, \\
	\mathcal{K}_\mu^+ &:= \left\{ \phi \in H^1 \backslash \{0\} \ : \ S_\mu(\phi)<S_\mu(Q_\mu), I_\mu(\phi)>0\right\}.
	\end{align}
	Using \eqref{S-I-intro} and the continuity argument, it is easy to see that the sets $\mathcal{K}_\mu^\pm$ are invariant under the flow of \eqref{NLS}. Note that
	\begin{align} \label{union}
	\mathcal{K}_\mu^- \cup \mathcal{K}_\mu^+ = \{\phi \in H^1 \backslash \{0\} \ : \ S_\mu(\phi)<S_\mu(Q_\mu)\} = \mathcal{A}_\mu^- \cup \mathcal{A}_\mu^+
	\end{align}
	since
	\[
	\left\{ \phi \in H^1 \backslash \{0\} \ : \ S_\mu(\phi)<S_\mu(Q_\mu), I_\mu(\phi)=0\right\} = \emptyset.
	\]
	By using variational arguments, we show (see Lemma $\ref{lem-est-I}$) that if $\phi \in H^1$ satisfies $E_\mu(\phi)\geq 0$ and $\phi \in \mathcal{K}^-_\mu$, then 
	\begin{align} \label{est-I-intro}
	I_\mu(\phi) \leq 2(S_\mu(\phi) - S_\mu(Q_\mu)).
	\end{align}
	Thanks to \eqref{est-I-intro}, the standard convexity argument of Glassey \cite{Glassey} implies the finite time blow-up for initial data in $\mathcal{K}^-_\mu$ satisfying some additional conditions. The result then follows by observing that $\mathcal{A}^-_\mu \equiv \mathcal{K}^-_\mu$ (see Lemma $\ref{lem-equi-A-K}$). We refer the reader to Section $\ref{S3}$ for more details.
	
	\begin{remark}
		We can construct an initial data $u_0 \in H^1 \cap L^2(|x|^2 dx)$ satisfying $\|\nabla u_0\|_{L^2} <1$ and $E_\mu(u_0)<0$ as follows. Let $\varphi \in H^1 \cap L^2(|x|^2 dx)$ be such that $\|\nabla \varphi\|_{L^2}<1$ (take for example $\varphi(x) = \frac{e^{-|x|}}{\sqrt{2\pi}}$). For $\lambda>0$, we denote $u_0(x)=\varphi(\lambda x)$. It follows that $\|\nabla u_0\|_{L^2} = \|\nabla \varphi\|_{L^2}<1$ and
		\[
		E_\mu(u_0) = \frac{1}{2} \|\nabla \varphi\|^2_{L^2} - \lambda^{-2} \int F_\mu(\varphi) dx <0
		\]
		provided 
		\[
		0<\lambda< \frac{\sqrt{2 \mathlarger{\int} F_\mu(\varphi) dx}}{\|\nabla \varphi\|_{L^2}}.
		\]
	\end{remark}
	
	\begin{remark}
		By Lemma $\ref{lem-equi-A-K}$ and Lemma $\ref{lem-example}$, there exists an initial data $u_0 \in H^1 \cap L^2(|x|^2 dx)$ satisfying $\|\nabla u_0\|_{L^2}<1$, $E_\mu(u_0) > 0$ and $u_0 \in \mathcal{A}^-_\mu$. 
	\end{remark}

	\begin{remark}
		We will see in Lemma $\ref{lem-posi-ener}$ that if $u_0 \in \mathcal{A}^+_\mu$, then $E_\mu(u_0) \geq 0$. Thus, by Theorem $\ref{theo-gwp}$ and Theorem $\ref{theo-blowup}$, we obtain sharp thresholds (within the radial or finite variance framework) for global existence and finite time blow-up for \eqref{NLS}.
	\end{remark}
	
	Our next result is the following energy scattering for \eqref{NLS} with radially symmetric initial data. 
	\begin{theorem} \label{theo-scat}
		Let $\mu=1$. Let $u_0 \in \mathcal{A}^+_1$ and $u_0$ be radially symmetric. Let $u$ be the corresponding global solution to \eqref{NLS}. Then the corresponding solution to \eqref{NLS} scatters in $H^1$ in both directions, i.e. there exists $u^{\pm} \in H^1$ such that
		\[
		\lim_{t\rightarrow \pm\infty} \|u(t) - e^{it\Delta} u^\pm\|_{H^1} =0.
		\]
	\end{theorem}
	
	Let us briefly describe the strategy of the proof. First, we show (see Lemma \ref{lem-open-set}) that 
	\begin{align} \label{defi-overline-A+}
	\overline{\Ac}^+_1:= \left\{\phi \in H^1 \ : \ S_1(\phi) < S_1(Q_1), P_1(\phi) \geq 0\right\} = \Ac^+_1 \cup \{0\}
	\end{align}
	is an open set of $H^1$. This is done by proving that functions with small $H^1$-norm belong to $\Ac^+_1$. To see this, we make use of a refined Moser-Trundinger inequality (see Corollary \ref{coro-refi-MT}) due to \cite{BIP}.
	
	Second, we prove that for data in $\Ac^+_1$, there exists $R_0=R_0(u_0,Q_1)>0$ sufficiently large such that the corresponding solution satisfies $\chi_{R} u(t) \in \overline{\Ac}^+_1$ for all $R\geq R_0$ and all $t\in \R$. Here $\chi_R$ is a suitable cutoff function.
	
	Third, thanks to the above observation and an argument using localized Morawetz estimates, we show that there exists $C=C(u_0,Q_1)>0$ such that
	\begin{align} \label{coer-est-intro}
	I_1(\chi_R u(t)) \geq C \|\chi_R u(t)\|^6_{L^6}
	\end{align}
	for all $R\geq R_0$ and all $t\in \R$. 
	
	Finally, using \eqref{coer-est-intro} and a modified argument of Arora-Dodson-Murphy's approach \cite{ADM}, we prove the global bound $\|u\|_{L^8(\R \times \R^2)} \leq C(u_0,Q_1)<\infty$ which yields the energy scattering. We refer the reader to Section \ref{S4} for more details.

	\begin{remark}
		It is expected that a same result holds for radial initial data in $\mathcal{A}^+_0$. In fact, most of results given in Section $\ref{S4}$ hold with $\mathcal{A}^+_0$ in place of $\mathcal{A}^+_1$. However, due to the critical nonlinearity $|u|^2u$ hidden in $f_0(u)$, we are not able to obtain similar scattering criteria as in Proposition $\ref{prop-scat}$. We hope to solve this proplem in a forthcoming work. 
	\end{remark}
	
	\begin{remark}
		In the preparation of this paper, we learnt that the energy scattering with radially symmetric initial data for the focusing NLS with exponential nonlinearity similar to $f_1$ was proved in a recent preprint \cite{GS}. The proof in \cite{GS} is also based on an argument of Arora-Dodson-Murphy \cite{ADM}, however, their proof is very different from ours. 
	\end{remark}
	
	The paper is organized as follows. In Section $\ref{S2}$, we recall some preliminaries needed in the paper such as Trundinger-Moser inequalities, the logarithmic inequality and Strichartz estimates. In Section $\ref{S3}$, we give the proofs of the global existence and finite time blow-up given in Theorem $\ref{theo-gwp}$ and Theorem $\ref{theo-blowup}$. Finally, in Section $\ref{S4}$, we give the proof of the long time dynamics for radially symmetric initial data given in Theorem $\ref{theo-scat}$.

	\section{Preliminaries}
	\label{S2}
	\setcounter{equation}{0}
	\subsection{Some useful inequalities}
	In this section, we recall some useful inequalities which are needed in the sequel. The first one is the following classical Moser-Trudinger inequality \cite{AT}.
	\begin{proposition}
		\label{prop-MT}
		Let $\alpha\in [0,4\pi)$. A constant $C_\alpha>0$ exists
		such that
		\begin{equation}
		\label{MT1}
		\int_{\R^2}\left(e^{\alpha |u|^2}-1\right)dx\leq C_\alpha
		\|u\|_{L^2}^2,
		\end{equation}
		for all $u\in H^1$ such that $\|\nabla
		u\|_{L^2} \leq 1$. Moreover, if $\alpha\geq 4\pi$, then \eqref{MT1} is false.
	\end{proposition}
	\begin{remark}
		\label{rem1} We point out that $\alpha=4\pi$ becomes admissible in \eqref{MT1} if we require $\|u\|_{H^1}\leq 1$ rather than
		$\|\nabla u\|_{L^2}\leq 1$. More precisely, we have
		\begin{equation}
		\label{MT2}
		\sup_{\|u\|_{H^1} \leq 1} \int_{\R^2} \left(e^{4\pi |u|^2}-1\right) dx=:\kappa<\infty,
		\end{equation}
		and this is	false for $\alpha>4\pi$ (see \cite{Ru} for more details). Here
		\[
		\|u\|_{H^1}^2=\|u\|_{L^2}^2+\|\nabla u\|_{L^2}^2.
		\]
	\end{remark}
	
	\begin{proposition} \label{prop-refi-MT}
		For all $u \in H^1$ with $\|\nabla u\|_{L^2} <1$, it holds that
		\begin{align} \label{refi-MT}
		\int_{\R^2} \left( e^{4\pi |u|^2} -1 \right) dx \leq \kappa \frac{\|u\|^2_{L^2}}{1-\|\nabla u\|^2_{L^2}},
		\end{align}
		where $\kappa$ is as in \eqref{MT2}.
	\end{proposition}
	
	\begin{proof}
		We fix $u\in H^1$ satisfying $\|\nabla u\|_2<1$ and define \footnote{We are grateful to Z. Guo for bringing our attention to the scaling argument in the proof.} $u^\lambda(x)=u(\lambda x)$ for some positive $\lambda$ to be chosen later. It follows that
		\[
		\|u^\lambda\|_{H^1}^2=\lambda^{-2}\|u\|_{L^2}^2+\|\nabla u\|_{L^2}^2.
		\]
		Choosing $\lambda$ such that $\|u^\lambda\|_{H^1}^2=1$, or equivalently
		\[
		\lambda^2=\frac{\|u\|_{L^2}^2}{1-\|\nabla u\|_{L^2}^2},
		\]
		it yields
		\[
		\int_{\R^2}\left(e^{4\pi|u^\lambda|^2}-1\right)dx\leq \kappa.
		\]
		Since
		\[
		\int_{\R^2}\left(e^{4\pi|u^\lambda|^2}-1\right)dx=\lambda^{-2}\int_{\R^2}\left(e^{4\pi|u|^2}-1\right)dx,
		\]
		we obtain \eqref{refi-MT} as desired.
	\end{proof}
	
	Thanks to \eqref{refi-MT}, we have the following global existence for \eqref{NLS}.
	
	\begin{lemma} \label{lem-glo-exi}
		Let $\mu \in \{0,1\}$. Let $u_0 \in H^1$ be such that $\|\nabla u_0\|_{L^2}<1$ and 
		\begin{align} \label{cond-global}
		\frac{\|u_0\|_{L^2}}{1-\|\nabla u_0\|^2_{L^2}} < \sqrt{\frac{\pi}{\kappa}},
		\end{align}
		where $\kappa$ is as in \eqref{MT2}. Then the corresponding solution to \eqref{NLS} exists globally in time.
	\end{lemma}
	
	\begin{proof}
		Let $T^*$ be the maximal forward time of existence. If $T^*=+\infty$, we are done. If $T^*<+\infty$, then \eqref{blow-alt} holds. Set 
		\[
		T_1:= \sup \left\{ 0 \leq t <T^* \ : \ \sup_{\tau \in [0,t]} \|\nabla u(t)\|^2_{L^2} <\frac{1}{2} \left( \|\nabla u_0\|^2_{L^2} +1 \right) \right\}.
		\]
		Since $\frac{1}{2} \left( \|\nabla u_0\|^2_{L^2} +1 \right) <1$, we infer from \eqref{blow-alt} that $T_1<T^*$. By the continuity of $t \mapsto \|\nabla u(t)\|^2_{L^2}$, we must have that 
		\begin{align} \label{est-T1}
		\|\nabla u(T_1)\|^2_{L^2} =\frac{1}{2}  \left( \|\nabla u_0\|^2_{L^2} +1 \right) .
		\end{align}
		By the conservation of energy,
		\begin{align*}
		\frac{1}{2} \|\nabla u(T_1)\|^2_{L^2} &= E_\mu(u(T_1)) + \int F_\mu(u(T_1)) dx \\
		&= E_\mu(u_0) + \int F_\mu(u(T_1)) dx \\
		&\leq \frac{1}{2} \|\nabla u_0\|^2_{L^2} + \int F_\mu(u(T_1)) dx.
		\end{align*}
		By \eqref{refi-MT}, \eqref{est-T1}, the conservation of mass and \eqref{cond-global}, we have
		\begin{align*}
		\int F_\mu(u(T_1)) dx &\leq \frac{1}{8\pi} \int \left( e^{4\pi|u(T_1)|^2} -1 \right) dx \\
		&\leq \frac{1}{8\pi} \kappa \frac{\|u(T_1)\|^2_{L^2}}{1-\|\nabla u(T_1)\|^2_{L^2}} \\
		&= \frac{1}{8\pi} \kappa \frac{\|u_0\|^2_{L^2}}{1-\|\nabla u(T_1)\|^2_{L^2}} \\
		&= \frac{\kappa}{4\pi} \frac{\|u_0\|^2_{L^2}}{1-\|\nabla u_0\|^2_{L^2}} \\
		&<\frac{1}{4}\left(1-\|\nabla u_0\|^2_{L^2}\right).
		\end{align*}
		It follows that
		\[
		\frac{1}{2} \|\nabla u(T_1)\|^2_{L^2}<\frac{1}{2} \|\nabla u_0\|^2_{L^2} +\frac{1}{4}\left(1-\|\nabla u_0\|^2_{L^2}\right) = \frac{1}{4} \left( \|\nabla u_0\|^2_{L^2}+1\right)
		\]
		which contradicts \eqref{est-T1}. The proof is complete.
	\end{proof}

	We also have the following refined Moser-Trudinger type inequalities due to \cite{BIP}.
	\begin{lemma} 
		Let $\alpha \in [0,4\pi)$ and $p \in(2,\infty)$. Then there exists $C_{\alpha,p}>0$ such that
		\begin{align} \label{MT-p}
		\int_{\R^2} e^{\alpha |u|^2} |u|^p dx \leq C_{\alpha,p} \|u\|^p_{L^p}
		\end{align}
		for all $u \in H^1$ satisfying $\|\nabla u\|_{L^2} \leq 1$.
	\end{lemma}
	
	\begin{corollary}
		Let $\alpha \in [0,4\pi)$. Then there exists $C_\alpha >0$ such that
		\begin{align} \label{MT-4}
		\int_{\R^2} \left(e^{\alpha |u|^2} -1 -\alpha |u|^2\right) dx \leq C_\alpha\|u\|^4_{L^4}
		\end{align}
		for all $u \in H^1$ satisfying $\|\nabla u\|_{L^2} \leq 1$.
	\end{corollary}
	
	\begin{corollary} \label{coro-refi-MT}
		Let $\beta \in (0,1)$ and $p \in (2,\infty)$. Then there exists $C_{\beta,p}>0$ such that 
		\begin{align} \label{MT-beta}
		\int_{\R^2} e^{4\pi |u|^2} |u|^p dx \leq C_{\beta,p} \|u\|^p_{L^p}
		\end{align}
		for all $u \in H^1$ satisfying $\|\nabla u\|_{L^2} \leq \beta$.
	\end{corollary}
	
	\begin{proof}
		Set $u_\beta= \frac{u}{\beta}$. Applying \eqref{MT-p} to $u_\beta$ and $\alpha= 4\pi \beta^2<4\pi$, we see that
		\[
		\int_{\R^2} e^{4\pi |u|^2} |u|^p dx  = \int_{\R^2} e^{4\pi \beta^2 |u_\beta|^2} \beta^p |u_\beta|^p dx \leq \beta^p C_{\beta,p} \|u_\beta\|^p_{L^p} = C_{\beta,p} \|u\|^p_{L^p}.
		\]
	\end{proof}
	
	\begin{corollary}
		Let $\beta \in (0,1)$ and $p \in (2,\infty)$. Then there exists $C_{\beta,p}>0$ such that 
		\begin{align} \label{MT-nu}
		\int_{\R^2} e^{4\pi(1+\nu) |u|^2} |u|^p dx \leq C_{\beta,p} \|u\|^p_{L^p}
		\end{align}
		for all $u\in H^1$ satisfying $\|\nabla u\|_{L^2} \leq \beta$ and $0<\nu<\frac{1}{\beta^2}-1$.
	\end{corollary}
	
	\begin{proof}
		Set $u_\beta:= \frac{u}{\beta}$. Applying \eqref{MT-p} to $u_\beta$ and $\alpha=4\pi(1+\nu) \beta^2 <4\pi$, we see that
		\[
		\int_{\R^2} e^{4\pi(1+\nu) |u|^2} |u|^p dx = \int_{\R^2} e^{4\pi(1+\nu) \beta^2 |u_\beta|^2} \beta^p |u_\beta|^p dx \leq \beta^p C_{\beta,p} \|u_\beta\|^p_{L^p} = C_{\beta,p} \|u\|^p_{L^p}.
		\]
	\end{proof}
	
	It is well-known that $H^1(\R^2)$ embeds continuously into $L^p(\R^2)$ for any $p \in (2,\infty)$ but not in $L^\infty(\R^2)$. However, one can estimate $L^\infty$-norm of functions in $H^1$ by a stronger norm but with a weaker growth. More precisely, we have the following logarithmic estimate due to \cite{IMM}.
	\begin{lemma} \label{lem-log}
		Let $0<\beta<1$. For any $\lambda >\frac{1}{2\pi \beta}$ and any $0<\omega \leq 1$, a constant $C_\lambda >0$ exists such that, for any function $u \in H^1 \cap \mathcal{C}^\beta$, 
		\begin{align} \label{log}
		\|u\|^2_{L^\infty} \leq \lambda \|u\|^2_{H_\omega} \log \left( C_\lambda + \frac{8^\beta \omega^{-\beta} \|u\|_{\mathcal{C}^\beta}}{\|u\|_{H_\omega}} \right),
		\end{align}
		where
		\[
		\|u\|^2_{H_\omega}:= \|\nabla u\|^2_{L^2} + \omega^2 \|u\|^2_{L^2}.
		\]
	\end{lemma}
	We also recall the following Sobolev embeddings which are needed in the paper (see e.g. \cite{BL}).
	
	\begin{proposition} \label{SobEmb}
		\begin{itemize}
			\item Let $1< p < q<\infty$ and $0< \gamma <\frac{2}{p}$ be such that $\frac{1}{q}=\frac{1}{p}-\frac{\gamma}{2}$. Then $\dot{W}^{\gamma,p}(\R^2) \hookrightarrow L^q(\R^2)$. In particular,
			\begin{equation*}
			\|u\|_{L^q}\lesssim \| |\nabla|^\gamma u\|_{L^p}.
			\end{equation*}
			\item Let $p>2$. Then $W^{1,p}(\R^2)\hookrightarrow \mathcal{C}^{1-\frac{2}{p}}(\R^2)$. 
			In particular
			\begin{align} \label{sobo-holder}
			W^{1,4}(\R^2) \hookrightarrow \mathcal{C}^{1/2}(\R^2).
			\end{align}
		\end{itemize}
	\end{proposition}
	
	
	The following continuity argument (or bootstrap argument) will also be useful for our purpose.
	\begin{lemma} \label{lem-boots}
		Let $I\subset\R$ be a time interval, and $X : I\to [0,\infty)$ be a continuous function satisfying, for every $t\in I$,
		\begin{equation}
		\label{boots1}
		X(t) \leq a + b [X(t)]^\theta,
		\end{equation}
		where $a,b>0$ and $\theta>0$ are constants. Assume that, for some $t_0\in I$,
		\begin{equation}
		\label{boots2}
		X(t_0)\leq 2a, \quad b <2^{-\theta}a^{1-\theta}.
		\end{equation}
		Then, for every $ t\in I$, we have
		\begin{equation}
		\label{boots3}
		X(t)\leq 2a.
		\end{equation}
	\end{lemma}
	\begin{proof}
		Assume there exists $t_1 \in I$ such that $X(t_1)>2a$. Then by continuity, there exists $t_2 \in [t_0,t_1)$ such that $X(t_2)=2a$. This contradicts the second assumption in \eqref{boots2} since $2a\leq a+ b(2a)^\theta$ implies $b\geq 2^{-\theta}a^{1-\theta}$.
	\end{proof}

	

	\subsection{Linear Schr\"odinger equation}
	It is well-known that solutions to the linear Schr\"odinger equation
	\[
	i\partial_t u + \Delta u =0, \quad u(0) = u_0, \quad (t,x) \in \R \times \R^2
	\]
	satisfy the $L^2$-isometry
	\begin{align} \label{L2-iso}
	\|u(t)\|_{L^2} = \|u_0\|_{L^2}
	\end{align}
	and for $t\ne 0$ the dispersive inequality
	\begin{align} \label{disper-est}
	\|u(t)\|_{L^\infty} \lesssim |t|^{-1} \|u_0\|_{L^1}.
	\end{align}
	\begin{definition}
		A pair $(q,r)$ is called {\bf Schr\"odinger admissible} if 
		\[
		q\in [2,\infty], \quad r \in [2,\infty), \quad \frac{2}{q}+\frac{2}{r}=1.
		\]
		A pair $(q,r)$ is called {\bf Schr\"odinger acceptable} if 
		\[
		q \in [1,\infty], \quad r\in [1,\infty), \quad (q,r) =(\infty,2) \quad \text{or} \quad \frac{1}{q}+\frac{2}{r}<1.
		\]
	\end{definition}
	Note that if $(q,r)$ is a Schr\"odinger admissible pair, then $(q,r)$ is also a Schr\"odinger acceptable pair. Thanks to \eqref{L2-iso}, \eqref{disper-est} and the $TT^*$-argument, we get the following Strichartz estimates (see \cite{Foschi, KT}).
	
	\begin{proposition}[Strichartz estimates \cite{Foschi,KT}] \label{prop-strichartz}
		There exists a positive constant $C$ such that the following estimates hold true:
		\begin{itemize}
			\item (Homogeneous estimates)
			\begin{align*}
			\|e^{it\Delta} u_0\|_{L^q(\R,L^r)} \leq C \|u_0\|_{L^2}
			\end{align*}
			for any Schr\"odinger admissible pair $(q,r)$.
			\item (Inhomogeneous estimates)
			\begin{align} \label{inho-stri-est}
			\left\|\int_0^t e^{i(t-s)\Delta} f(s) ds \right\|_{L^q(\R,L^r)} \leq C \|f\|_{L^{m'}(\R,L^{n'})}
			\end{align}
			for any Schr\"odinger acceptable pairs $(q,r)$ and $(m,n)$ satisfying 
			\[
			\frac{2}{q}+\frac{2}{r} = 2- \left(\frac{2}{m}+\frac{2}{n}\right).
			\]
		\end{itemize}
	\end{proposition}
	Note that, in particular, $(4,4)$ is a Schr\"odinger admissible pair, and its dual pair is $\left(\frac{4}{3},\frac{4}{3}\right)$. We denote for any time interval $I \subset \R$, 
	\begin{align} \label{defi-ST-norm}
	\|u\|_{\ST(I)}:=  \|\scal{\nabla} u\|_{L^4(I\times \R^2)} + \|\scal{\nabla} u\|_{L^\infty(I,L^2)}
	\end{align}
	and
	\begin{align} \label{defi-ST-dual}
	\|u\|_{\ST^*(I)}:=  \|\scal{\nabla} u\|_{L^{\frac{4}{3}}(I\times \R^2)},
	\end{align}
	where $\scal{\nabla}=\sqrt{1-\Delta}$.

	\section{Sharp thresholds for global existence and blow-up}
	\label{S3}
	\setcounter{equation}{0}
	\subsection{Global existence}
	Let us prove the global existence for initial data in $\mathcal{A}^+_\mu$. 
	
	\noindent {\bf Proof of Theorem $\ref{theo-gwp}$.}
	We first note that if $u_0 \in \mathcal{A}^+_\mu$, then 
	\[
	\frac{1}{2} \|\nabla u_0\|^2_{L^2} = S_\mu(u_0) - P_\mu(u_0) < S_\mu(u_0) < S_\mu(Q_\mu) = \frac{1}{2} \|\nabla Q_\mu\|^2_{L^2} <\frac{1}{2}
	\]
	which implies that $\|\nabla u_0\|_{L^2}<1$. By Theorem $\ref{theo-lwp}$, there exists a unique local solution to \eqref{NLS} with initial data $u_0$. Let $[0,T^*)$ be the maximal forward time interval of existence. Since $\mathcal{A}^+_\mu$ is invariant under the flow of \eqref{NLS}, we have $u(t) \in \mathcal{A}^+_\mu$ for all $t\in [0,T^*)$. By the conservation of mass and energy, we see that
	\[
	\frac{1}{2} \|\nabla u(t)\|^2_{L^2} = S_\mu(u(t)) - P_\mu(u(t)) < S_\mu(u(t))= S_\mu(u_0) < S_\mu(Q_\mu)
	\]
	for all $t\in [0,T^*)$. Hence $\|\nabla u(t)\|_{L^2} < 2S_\mu(Q_\mu)<1$ for all $t\in [0,T^*)$. By the local theory, we can extend the local solution globally in time. 
	\hfill $\Box$
	
	\subsection{Finite time blow-up}
	In this subsection, we give the proof of the finite time blow-up given in Theorem $\ref{theo-blowup}$. We will consider separately two cases: $E_\mu(u_0)<0$ and $E_\mu(u_0) \geq 0$.
	
	\subsubsection{\bf Finite time blow-up for negative energy initial data}
	
	\begin{lemma} \label{lem-blow-nega-Sig}
		Let $\mu \in \{0,1\}$. Let $u_0 \in H^1$ be such that $\|\nabla u_0\|_{L^2}<1$ and $E_\mu(u_0)<0$. If $u_0 \in L^2(|x|^2 dx)$, then the corresponding solution to \eqref{NLS} blows up in finite time.
	\end{lemma}

	\begin{proof}
		By the local theory, the condition $\|\nabla u_0\|_{L^2}<1$ ensures the existence of local solutions for \eqref{NLS}. It is well-known (see e.g. \cite{Cazenave}) that if $u_0 \in \Sigma:= H^1 \cap L^2(|x|^2 dx)$, then the corresponding solution belongs to $\Sigma$ and satisfies
		\[
		\frac{d^2}{dt^2} \|xu(t)\|^2_{L^2} =8 I_\mu(u(t)) = 16 E_\mu(u(t)) - 8 \int \overline{u}(t) f_\mu(u(t)) - 4 F_\mu(u(t)) dx
		\]
		for all $t\in [0,T^*)$.	Note that a direct computation shows
		\[
		\overline{u} f_\mu(u) - 4 F_\mu(u) = \frac{1}{4\pi} g_\mu(4\pi |u|^2),
		\]
		where
		\[
		g_\mu(s):= s(e^s-1-\mu s) - 2 \left(e^s -1 - s -\frac{\mu}{2} s^2 \right).
		\]
		It is easy to see that $g_\mu(s) \geq 0$ for all $s\geq 0$, hence 
		\begin{align} \label{posi-nonli}
		\int \overline{u} f_\mu(u) - 4 F_\mu(u) dx\geq 0
		\end{align}
		for all $u \in H^1$. This together with the conservation of energy imply that
		\[
		\frac{d^2}{dt^2} \|xu(t)\|^2_{L^2} \leq 16 E_\mu(u(t)) = 16 E_\mu(u_0) <0
		\]
		for all $t\in [0,T^*)$. By the standard argument of Glassey \cite{Glassey}, the solution must blow up in finite time.
	\end{proof}

	We are next interested in the finite time blow-up of radial solutions for \eqref{NLS}. 
	To do this, we need the following virial estimates related to \eqref{NLS}. Given a real-valued function $\varphi$, we define the virial potential
	\[
	V_\varphi(t):= \int \varphi |u(t)|^2 dx.
	\]
	\begin{lemma} [\cite{TVZ}] Let $u$ be a sufficiently smooth and decaying solution to \eqref{NLS}. Then it holds that
		\[
		\frac{d}{dt} V_\varphi(t) = 2 \int \nabla \varphi \cdot \emph{Im}(\overline{u}(t) \nabla u(t)) dx
		\]
		and
		\begin{align*} 
		\frac{d^2}{dt^2} V_\varphi(t) = - \int \Delta^2 \varphi |u(t)|^2 dx &+ 4\sum_{j,k} \int \partial^2_{jk} \varphi \emph{ Re}(\partial_j \overline{u}(t) \partial_k u(t)) dx \\
		&- 2 \int \nabla \varphi \cdot \{f(u), u\}_p(t) dx,
		\end{align*}
		where $\{f,g\}_p = \emph{Re}(f \nabla \overline{g}-g \nabla \overline{f})$ is the momentum bracket.
	\end{lemma}
	
	In our case $f_\mu(u) = \left( e^{4\pi |u|^2} -1 - 4\pi \mu|u|^2\right) u$, a direct computation shows that
	\begin{align*}
	\{f_\mu(u),u\}_p &= \text{Re} \left(\left( e^{4\pi |u|^2}-1 - 4\pi \mu |u|^2\right) u \nabla \overline{u} - u \nabla \left( \left(e^{4\pi |u|^2}-1 - 4\pi \mu |u|^2 \right) \overline{u} \right) \right) \\
	&= - |u|^2 \nabla \left(e^{4\pi |u|^2}-1 - 4\pi \mu |u|^2\right) \\
	&= -\frac{1}{4\pi} \nabla \left( h_\mu(4\pi |u|^2)\right),
	\end{align*}
	where 
	\begin{align} \label{defi-h-mu}
	h_\mu(s):= s(e^s-1-\mu s) - \left( e^s -1-s-\frac{\mu}{2} s^2\right) = se^s - e^s +1 - \frac{\mu}{2} s^2.
	\end{align}
	We thus obtain
	\begin{align} \label{seco-viri}
	\frac{d^2}{dt^2} V_\varphi(t) = - \int \Delta^2 \varphi |u(t)|^2 dx + 4\sum_{j,k} \int \partial^2_{jk} \varphi \text{ Re}(\partial_j \overline{u}(t) \partial_k u(t)) dx - 2 \int \Delta \varphi H_\mu(u(t)) dx,
	\end{align}
	where
	\begin{align} \label{defi-H-mu}
	H_\mu(u) := \frac{1}{4\pi} h_\mu(4\pi |u|^2)  = \overline{u} f_\mu(u) - 2F_\mu(u).
	\end{align}
	Taking $\varphi(x) = |x|^2$ and using the fact $\partial^2_{jk}\varphi= 2 \delta_{jk}$ and $\Delta \varphi= 4$, we get
	\begin{align} \label{virial-proof}
	\frac{d^2}{dt^2} \|xu(t)\|^2_{L^2} = 8 \left(\|\nabla u(t)\|^2_{L^2} - \int H_\mu(u(t)) dx \right) = 8 I_\mu(u(t))
	\end{align}
	which confirms \eqref{virial}.
	
	Let $\zeta: [0,\infty) \rightarrow [0,2]$ be a smooth function satisfying
	\begin{align*} 
	\zeta(r) = \left\{
	\begin{array}{ccl}
	2 & \text{if} & 0\leq r \leq 1, \\
	0 &\text{if} & r\geq 2.
	\end{array}
	\right.
	\end{align*}
	We define the function $\theta: [0,\infty) \rightarrow [0,\infty)$ by
	\[
	\theta(r):= \int_0^r \int_0^s \zeta(z)dz ds.
	\]
	For $R>0$, we define the radial function
	\[
	\varphi_R(x) = \varphi_R(r) := R^2 \theta(r/R), \quad r=|x|.
	\]
	It is easy to see that
	\begin{align} \label{varphi-R}
	2 \geq \varphi''_R(r) \geq 0, \quad 2 - \frac{\varphi'_R(r)}{r} \geq 0, \quad 4-\Delta \varphi_R(x)\geq 0, \quad \forall r \geq 0, \quad \forall x \in \R^2.
	\end{align}
	\begin{lemma} \label{lem-viri-est}
		Let $\mu \in \{0,1\}$. Let $u_0 \in H^1$ be radially symmetric and satisfy $\|\nabla u_0\|_{L^2}<1$. Then the corresponding solution to \eqref{NLS} satisfies for any $t\in [0,T^*)$,
		\begin{align} \label{viri-est}
		\frac{d^2}{dt^2} V_{\varphi_R}(t) \leq 8 I_\mu(u(t)) +  CR^{-2} + C\left(e^{C R^{-1}} -1\right)
		\end{align}
		for some constant $C$ independent of $R$.
	\end{lemma}
	
	\begin{proof}
		Since $\varphi_R(x) =|x|^2$ for $|x| \leq R$, we see that
		\begin{align*}
		\frac{d^2}{dt^2} V_{\varphi_R}(t) &= 8 \left(\|\nabla u(t)\|^2_{L^2} - \int H_\mu(u(t)) dx \right) \\
		&\mathrel{\phantom{=}} - 8\|\nabla u(t)\|^2_{L^2(|x|>R)} + 4 \sum_{j,k} \int_{|x|>R} \partial^2_{jk} \varphi_R \text{Re} (\partial_j \overline{u}(t) \partial_k u(t)) dx \\
		&\mathrel{\phantom{=}} - \int_{|x|>R} \Delta^2 \varphi_R |u(t)|^2 dx + 2 \int_{|x|>R} (4-\Delta \varphi_R) H_\mu(u(t)) dx.
		\end{align*}
		Since $u$ is radial, we use \eqref{varphi-R} and the fact
		\[
		\partial^2_{jk} = \left(\frac{\delta_{jk}}{r} - \frac{x_jx_k}{r^3}\right) \partial_r + \frac{x_j x_k}{r^2} \partial^2_r
		\]
		to get that
		\[
		\sum_{j,k} \partial^2_{jk} \varphi_R \partial_j \overline{u} \partial_k u = \varphi''_R(r) |\partial_r u|^2 \leq 2 |\partial_r u|^2 = 2 |\nabla u|^2.
		\]
		Thus
		\[
		4 \sum_{j,k} \int_{|x|>R} \partial^2_{jk} \varphi_R \text{Re}(\partial_j \overline{u} \partial_k u) dx - 8 \|\nabla u\|^2_{L^2(|x|>R)} \leq 0.
		\]
		Since $|\Delta^2\varphi_R| \lesssim R^{-2}$ and $|4-\Delta \varphi_R| \lesssim 1$, we have that
		\[
		\frac{d^2}{dt^2} V_{\varphi_R}(t) \leq 8 I_\mu(u(t)) + CR^{-2} + \int_{|x|>R} H_\mu(u(t)) dx.
		\]
		To estimate the last term, we recall the following radial Sobolev embedding due to Strauss \cite{Strauss}: 
		\[
		\sup_{x \ne 0} |x|^{\frac{1}{2}} |f(x)| \leq C \|\nabla f\|_{L^2}^{\frac{1}{2}} \|f\|_{L^2}^{\frac{1}{2}}, \quad \forall f \in H^1_{\text{rad}}.
		\]
		Note that 
		\[
		h_\mu(s) = se^s - e^s +1 - \frac{\mu}{2}s^2 \leq se^s-e^s+1 \leq s(e^s-1)
		\]
		for all $s\geq 0$ which implies that $H_\mu(u) \lesssim |u|^2 \left(e^{4\pi |u|^2}-1 \right)$. By the conservation of mass, the fact $\sup_{t\in [0,T^*)} \|\nabla u(t)\|_{L^2} \leq 1$ and the radial Sobolev embedding, we see that
		\begin{align} \label{est-non}
		\int_{|x|>R} H_\mu(u(t)) dx \lesssim \left(e^{4\pi \|u(t)\|_{L^\infty}^2}-1 \right) \|u(t)\|^2_{L^2} \lesssim e^{CR^{-1}} -1
		\end{align}
		for all $t\in [0,T^*)$. The proof is complete.
	\end{proof}
	

	\begin{lemma} \label{lem-blow-nega-radi}
		Let $\mu \in \{0,1\}$. Let $u_0 \in H^1$ be such that $\|\nabla u_0\|_{L^2} <1$ and $E_\mu(u_0)<0$. If $u_0$ is radially symmetric, then the corresponding solution to \eqref{NLS} blows up in finite time.
	\end{lemma}
	
	\begin{proof}
		Since $u_0 \in H^1_{\rad}$ satisfies $\|\nabla u_0\|_{L^2} <1$, the local solution to \eqref{NLS} with initial data $u_0$ belongs to $H^1_{\rad}$ for all $t\in[0,T^*)$ and satisfies $\sup_{t\in [0,T^*)} \|\nabla u(t)\|_{L^2} \leq 1$. Applying Lemma $\ref{lem-viri-est}$, we have for any $t\in [0,T^*)$,
		\[
		\frac{d^2}{dt^2} V_{\varphi_R}(t) \leq 8 I_\mu(u(t)) + CR^{-2} + C\left(e^{CR^{-1}}-1\right)
		\]
		for some $C>0$ independent of $R$. By \eqref{posi-nonli} and the conservation of energy, we get
		\[
		\frac{d^2}{dt^2} V_{\varphi_R}(t) \leq 16 E_\mu(u_0) + CR^{-2} + C\left(e^{CR^{-1}}-1\right)
		\]
		for all $t\in [0,T^*)$. Taking $R>0$ sufficiently large, we obtain
		\[
		\frac{d^2}{dt^2} V_{\varphi_R}(t) \leq 8 E_\mu(u_0) <0
		\]
		for all $t\in [0,T^*)$. The standard argument of Glassey yields $T^*<+\infty$.
	\end{proof}

	\subsubsection{\bf Finite time blow-up for non-negative energy initial data}
	
	\begin{lemma} \label{lem-I}
		Let $\mu \in \{0,1\}$ and $\phi \in H^1\backslash \{0\}$. 
		Denote 
		\begin{align} \label{defi-phi-lamb}
		\phi_\lambda(x):= \lambda \phi(\lambda x).
		\end{align}
		It holds that
		\begin{itemize}
			\item $\partial_\lambda S_\mu(\phi_\lambda) = \frac{1}{\lambda} I_\mu(\phi_\lambda)$.
			\item $I_\mu (\phi_\lambda) = \lambda^2 \Phi_\mu(\lambda)$, where
			\[
			\lambda \mapsto \Phi_\mu (\lambda)
			\]
			is a strictly decreasing function on $(0,\infty)$ and
			\[
			\Phi_\mu(0):= \lim_{\lambda \rightarrow 0} \Phi_\mu(\lambda) = \|\nabla \phi\|^2_{L^2} - 2\pi(1-\mu) \|\phi\|^4_{L^4}.
			\]
		\end{itemize}
		In particular, if there exists ${\lambda_\mu }>0$ such that
		$\Phi_\mu(\lambda_\mu)=0$ then $\lambda_\mu$ is unique and 
		$\lambda\mapsto S_\mu(\phi_\lambda)$ has a strict maximum in $\lambda_\mu$.
	\end{lemma}

	\begin{remark} \label{rem-I}
		The existence of such ${\lambda_\mu }$ is equivalent to $\Phi_\mu(0)>0$. This is due to the fact that $\lim_{\lambda \rightarrow \infty} \Phi_\mu(\lambda) =-\infty$. For $\mu=1$, it is always true. For $\mu=0$, it is true under the  following sufficient (not necessary) condition: $E_0(\phi)\geq 0$. Indeed, we write
		\begin{align*}
		E_0(\phi) &= \frac{1}{2} \left(\|\nabla \phi\|^2_{L^2} -2\pi \|\phi\|^4_{L^4} \right) - \frac{1}{8\pi} \int \left(e^{4\pi|\phi|^2} - 1- 4\pi |\phi|^2 - 8\pi^2 |\phi|^4\right) dx \\ 
		&= \frac{1}{2} \Phi_0(0) - \frac{1}{8\pi} \int \left(e^{4\pi|\phi|^2} - 1- 4\pi |\phi|^2 - 8\pi^2 |\phi|^4\right) dx.
		\end{align*}
		The condition $E_0(\phi) \geq 0$ implies that
		\[
		\Phi_0(0) \geq \frac{1}{4\pi} \int \left(e^{4\pi|\phi|^2} - 1- 4\pi |\phi|^2 - 8\pi^2 |\phi|^4\right) dx.
		\]
		Since $\phi \ne 0$, we infer that $\Phi_0(0)>0$.
	\end{remark}
	
	
	\noindent {\it Proof of Lemma $\ref{lem-I}$.}
		We have 
		\begin{align*}
		I_\mu(\phi_\lambda) &= \|\nabla \phi_\lambda\|^2_{L^2} - \int \overline{\phi}_\lambda f_\mu(\phi_\lambda) - 2 F_\mu(\phi_\lambda) dx \\
		&= \lambda^2 \|\nabla \phi\|^2_{L^2} - \int \lambda \overline{\phi}(\lambda x) f_\mu(\lambda \phi(\lambda x)) - 2F_\mu(\lambda \phi(\lambda x)) dx \\
		&= \lambda^2 \|\nabla \phi\|^2_{L^2} - \lambda^{-2} \int \lambda \overline{\phi} f_\mu(\lambda \phi)  - 2F_\mu(\lambda \phi) dx.
		\end{align*}
		Similarly, 
		\[
		S_\mu(\phi_\lambda) = \frac{\lambda^2}{2} \|\nabla \phi\|^2_{L^2} +\frac{1}{2} \|\phi\|^2_{L^2} - \lambda^{-2} \int F_\mu(\lambda \phi) dx.
		\]
		We compute
		\[
		\partial_\lambda S_\mu(\phi_\lambda) = \lambda \|\nabla \phi\|^2_{L^2} + 2\lambda^{-3} \int F_\mu(\lambda \phi) - \lambda^{-2} \int \overline{\phi} f_\mu(\lambda \phi) dx.
		\]
		It follows that $\partial_\lambda S_\mu(\phi_\lambda) = \frac{1}{\lambda} I_\mu(\phi_\lambda)$. Using the fact
		\[
		\lambda \overline{\phi} f_\mu(\lambda \phi) - 2 F_\mu(\lambda \phi) = \frac{1}{4\pi} h_\mu(4\pi \lambda^2 |\phi|^2)
		\]
		with $h_\mu$ as in \eqref{defi-h-mu}, we write
		\[
		I_\mu(\phi_\lambda) = \lambda^2 \left( \|\nabla \phi\|^2_{L^2} - \lambda^{-4} \int \frac{1}{4\pi}h_\mu(4\pi \lambda^2 |\phi|^2) dx \right).
		\]
		We next write
		\[
		\lambda^{-4} h_\mu(4\pi \lambda^2 |\phi|^2) = 16 \pi^2 |\phi|^4 k_\mu(4\pi \lambda^2 |\phi|^2),
		\]
		where 
		\[
		k_\mu(s) := s^{-2} h_\mu(s) = s^{-1} e^s - s^{-2} e^s + s^{-2} -\frac{\mu}{2}.
		\]
		Note that if we set $l_\mu(s):= s^3 k'_\mu(s)$, then $l'_\mu(s) = s^2 e^s \geq 0$. Since $l_\mu(0)=0$, we get $l_\mu(s)\geq 0$ hence $k'_\mu(s) \geq 0$ for all $s\geq 0$. It follows that $k_\mu$ is a strictly increasing function on $(0,\infty)$. We infer that
		\begin{align} \label{defi-Phi}
		\lambda \mapsto \|\nabla \phi\|^2_{L^2} - \int 4\pi |\phi|^4 k_\mu(4\pi \lambda^2 |\phi|^2) dx =: \Phi_\mu(\lambda)
		\end{align}
		is strictly decreasing on $(0,\infty)$. A direct computation shows
		\[
		\Phi_\mu(0) := \lim_{\lambda \rightarrow 0} \Phi_\mu(\lambda) = \|\nabla \phi\|^2_{L^2} - 2\pi(1-\mu)\|\phi\|^4_{L^4}
		\]	
		and $\lim_{\lambda \rightarrow \infty} \Phi_\mu(\lambda) = -\infty$ since $k_\mu(s)\geq \frac{1}{3}s$ for all $s\geq 0$.	The other properties follow easily, and the proof is complete.
	\hfill $\Box$

	\begin{lemma} \label{lem-S-I}
		Let $\mu \in \{0,1\}$. It holds that
		\begin{align} \label{S-I}
		S_\mu(Q_\mu)= \inf \left\{ S_\mu(\phi) \ : \ \phi \in H^1 \backslash \{0\}, I_\mu(\phi)=0 \right\}.
		\end{align}
	\end{lemma}
	
	\begin{proof}
		Denote 
		\[
		N_\mu:= \inf \left\{ S_\mu(\phi) \ : \ \phi \in H^1 \backslash \{0\}, I_\mu(\phi)=0\right\}.
		\]
		It is clear that $S_\mu(Q_\mu) \geq N_\mu$ since $I_\mu(Q_\mu)=0$. Let $\phi \in H^1 \backslash \{0\}$ be such that $I_\mu(\phi)=0$. It follows from Lemma $\ref{lem-I}$ that $\lambda_\mu=1$ and $S_\mu(\phi) \geq S_\mu(\phi_\lambda)$ for all $\lambda>0$. If $P_\mu(\phi)=0$, then $S_\mu(\phi) \geq S_\mu(Q_\mu)$ due to \eqref{chara-P}. If $P_\mu(\phi) \ne 0$, then we have from the fact that the map $\lambda \mapsto P_\mu(\phi_\lambda)$ is strictly decreasing on $(0,\infty)$, $\lim_{\lambda \rightarrow 0} P_\mu(\phi_\lambda)= \frac{1}{2} \|\phi\|^2_{L^2}>0$ and $\lim_{\lambda \rightarrow \infty} P_\mu(\phi_\lambda) = -\infty$ that there exists $\tilde{\lambda}_\mu>0$ such that $P_\mu(\phi_{\tilde{\lambda}_\mu})=0$. Thus, by \eqref{chara-P} and the fact $\lambda \mapsto S_\mu (\phi_\lambda)$ attains its maximum at $\lambda=1$, we see that $S_\mu(\phi) \geq S_\mu(\phi_{\tilde{\lambda}_\mu}) \geq S_\mu(Q_\mu)$. In both cases, we have $S_\mu(\phi) \geq S_\mu(Q_\mu)$. Taking the infimum over all $\phi \in H^1\backslash \{0\}$ satisfying $I_\mu(\phi)=0$, we get $N_\mu \geq S_\mu(Q_\mu)$ hence $N_\mu = S_\mu(Q_\mu)$. 
	\end{proof}


	\begin{lemma} \label{lem-est-I}
		Let $\mu \in \{0,1\}$ and $\phi \in H^1$ be such that $E_\mu(\phi)\geq 0$. If $\phi \in \mathcal{K}^-_\mu$, then it holds that
		\begin{align} \label{est-I}
		I_\mu(\phi) \leq 2 \left(S_\mu(\phi) - S_\mu(Q_\mu)\right).
		\end{align}
	\end{lemma}
	
	\begin{proof}
		Let $\phi_\lambda$ be as in Lemma $\ref{lem-I}$. 
		\[
		S_\mu(\phi_\lambda) = \frac{\lambda^2}{2} \|\nabla \phi\|^2_{L^2} + \frac{1}{2}\|\phi\|^2_{L^2} - \lambda^{-2} \int F_\mu(\lambda \phi) dx.
		\]
		We have
		\begin{align} \label{deri-S}
		\partial_\lambda S_\mu(\phi_\lambda) = \lambda\|\nabla \phi\|^2_{L^2} - \lambda^{-3} \int \lambda \overline{\phi} f_\mu(\lambda \phi) - 2 F_\mu(\lambda \phi) dx = \frac{I_\mu(\phi_\lambda)}{\lambda}.
		\end{align}
		We also have
		\begin{align*}
		\partial_\lambda I_\mu(\phi_\lambda) &= 2\lambda \|\nabla \phi\|^2_{L^2} + 2 \lambda^{-3} \int \lambda \overline{\phi} f_\mu(\lambda \phi) - 2F_\mu(\lambda\phi) dx \\
		&\mathrel{\phantom{=2\lambda \|\nabla \phi\|^2_{L^2}}} - \lambda^{-2} \int \partial_\lambda[\lambda \overline{\phi} f_\mu(\lambda \phi) - 2F_\mu(\lambda \phi)] dx \\
		&=2\left(\lambda\|\nabla \phi\|^2_{L^2} - \lambda^{-3} \int \lambda \overline{\phi} f_\mu(\lambda \phi) - 2 F_\mu(\lambda \phi) dx \right) \\
		&\mathrel{\phantom{=}} + 4 \lambda^{-3} \int \lambda \overline{\phi} f_\mu(\lambda \phi) - 2F_\mu(\lambda \phi) dx - \lambda^{-2} \int \partial_\lambda[\lambda \overline{\phi} f_\mu(\lambda \phi) - 2F_\mu(\lambda \phi)] dx \\
		&= 2 \partial_\lambda S_\mu(\phi_\lambda) +  \left[ 4 \lambda^{-3} \int \lambda \overline{\phi} f_\mu(\lambda \phi) - 2F_\mu(\lambda \phi) dx - \lambda^{-2} \int \partial_\lambda[\lambda \overline{\phi} f_\mu(\lambda \phi) - 2F_\mu(\lambda \phi)] dx \right].
		\end{align*}
		A direct computation shows that the term inside the square bracket becomes
		\[
		\pi^{-1} \lambda^{-3} \int m_\mu(4\pi \lambda^2 |\phi|^2) dx,
		\]
		where
		\[
		m_\mu(s):= h_\mu(s) - \frac{s}{2} h'_\mu(s) = se^s - e^s - \frac{s^2}{2} e^s +1
		\]
		with $h_\mu$ as in \eqref{defi-h-mu}. Note that
		\[
		m'_\mu(s) = -\frac{s^2}{2}e^s \leq 0
		\]
		for all $s\geq 0$ which implies $m_\mu(s) \leq m_\mu(0) =0$ for all $s\geq 0$. It follows that
		\begin{align} \label{est-I-mu}
		\partial_\lambda I_\mu(\phi_\lambda) \leq 2 \partial_\lambda S_\mu(\phi_\lambda), \quad  \forall \lambda>0.
		\end{align}
		Since $I_\mu(\phi) <0$, we see that $\Phi_\mu(1)<0$ (see \eqref{defi-Phi}). On the other hand, $\Phi_\mu(0)>0$ by Remark $\ref{rem-I}$. Thus there exists $\tilde{\lambda}_\mu \in (0,1)$ such that $\Phi_\mu(\tilde{\lambda}_\mu) =0$ or $I_\mu (\phi_{\tilde{\lambda}_\mu}) =0$. Integrating \eqref{est-I-mu} over the interval $(\tilde{\lambda}_\mu,1)$, we obtain
		\[
		I_\mu(\phi) = I_\mu(\phi) - I_\mu(\phi_{\tilde{\lambda}_\mu}) \leq 2 \left(S_\mu(\phi) - S_\mu(\phi_{\tilde{\lambda}_\mu})\right) \leq 2 \left(S_\mu(\phi) - S_\mu(Q_\mu)\right).
		\]
		Here we have used the fact $S_\mu(\phi_{\tilde{\lambda}_\mu}) \geq S_\mu(Q_\mu)$ which follows from \eqref{S-I} and $I_\mu(\phi_{\tilde{\lambda}_\mu})=0$. The proof is complete.
	\end{proof}
	
	\begin{lemma} \label{lem-blowup-Sigma}
		Let $\mu \in \{0,1\}$. Let $u_0 \in H^1$ be such that $\|\nabla u_0\|_{L^2} <1$ and $E_\mu(u_0)\geq 0$. If $u_0 \in \mathcal{K}^-_\mu$ and $u_0 \in \Sigma:= H^1 \cap L^2(|x|^2 dx)$, then the corresponding solution to \eqref{NLS} blows up in finite time.
	\end{lemma}
	
	\begin{proof}
		Since $\mathcal{K}^-_\mu$ is invariant under the flow of \eqref{NLS}, we have $u(t) \in \mathcal{K}^-_\mu$ for all $t \in [0,T^*)$. Moreover, by the conservation of energy, $E_\mu(u(t)) = E_\mu(u_0) \geq 0$ for all $t\in [0,T^*)$. By \eqref{virial} and Lemma $\ref{lem-est-I}$,
		\[
		\frac{d^2}{dt^2} \|xu(t)\|^2_{L^2} =8 I_\mu(u(t)) \leq 16 \left( S_\mu(u(t)) - S_\mu(Q_\mu)\right) = 16 \left(S_\mu(u_0) - S_\mu(Q_\mu)\right) <0
		\]
		for all $t\in [0,T^*)$. This shows that the solution must blow up in finite time.
	\end{proof}
		
	\begin{lemma} \label{lem-blowup-rad}
		Let $\mu \in \{0,1\}$. Let $u_0 \in H^1$ be such that $\|\nabla u_0\|_{L^2} <1$ and $E_\mu(u_0)\geq 0$. If $u_0 \in \mathcal{K}^-_\mu$ and $u_0$ is radially symmetric, then the corresponding solution to \eqref{NLS} blows up in finite time.
	\end{lemma}
	\begin{proof}
		Applying Lemma $\ref{lem-viri-est}$, we get for any $t\in [0,T^*)$,
		\[
		\frac{d^2}{dt^2} V_{\varphi_R}(t) \leq 8 I_\mu(u(t)) + CR^{-2} +C\left(e^{C R^{-1}} -1\right)
		\]
		for some $C>0$ independent of $R$. Since $\mathcal{K}^-_\mu$ is invariant under the flow of \eqref{NLS}, we have that $u(t) \in \mathcal{K}^-_\mu$ for all $t\in [0,T^*)$. Moreover, by the conservation of energy, $E_\mu(u(t)) = E_\mu(u_0)\geq 0$ for all $t\in [0,T^*)$. We thus apply Lemma $\ref{lem-est-I}$ and the conservation of mass and energy to get
		\begin{align*}
		\frac{d^2}{dt^2} V_{\varphi_R}(t) &\leq 16 (S_\mu(u(t)) - S_\mu(Q_\mu)) + CR^{-2} +C\left(e^{C R^{-1}} -1\right) \\
		&= 16 (S_\mu(u_0) - S_\mu(Q_\mu)) + CR^{-2} +C\left(e^{C R^{-1}} -1\right)
		\end{align*}
		for all $t\in [0,T^*)$. By taking $R>0$ sufficiently large, we obtain
		\[
		\frac{d^2}{dt^2} V_{\varphi_R}(t) \leq 8 (S_\mu(u_0) - S_\mu(Q_\mu))<0
		\]
		for all $t\in [0,T^*)$. This implies that the solution blows up in finite time.
	\end{proof}
	
	\begin{lemma} \label{lem-posi-ener}
		Let $\mu \in \{0,1\}$. If $\phi \in \mathcal{A}^+_\mu$, then $E_\mu(\phi) \geq 0$.
	\end{lemma}
	
	\begin{proof}
		We firstly note that 
		\[
		\mathcal{A}^+_\mu \cap \left\{ \phi \in H^1 \backslash \{0\} \ : \ E_\mu(\phi)<0\right\} \cap \Sigma = \emptyset
		\]
		due to Theorem $\ref{theo-gwp}$ and Lemma $\ref{lem-blow-nega-Sig}$. The result follows by noticing that $\mathcal{A}^+_\mu$ and $\{ \phi \in H^1 \backslash \{0\} \ : \ E_\mu(\phi)<0\}$ are open sets of $H^1$ and $\Sigma$ is dense in $H^1$.
	\end{proof}

	We also have the following equivalence of invariant sets.
	\begin{lemma} \label{lem-equi-A-K}
		Let $\mu \in \{0,1\}$. It holds that $\mathcal{A}^+_\mu \equiv \mathcal{K}^+_\mu$ and $\mathcal{A}^-_\mu \equiv \mathcal{K}^-_\mu$.
	\end{lemma}
	
	\begin{proof}
		Thanks to \eqref{union}, it suffices to show $\mathcal{K}^+_\mu \equiv \mathcal{A}^+_\mu$. 
		Let us prove the first inclusion $\mathcal{K}^+_\mu \subset \mathcal{A}^+_\mu$. Let $\phi \in \mathcal{K}^+_\mu$, i.e. $S_\mu(\phi) <S_\mu(Q_\mu)$ and $I_\mu(\phi)>0$. We need to show that $P_\mu(\phi)>0$. Assume by contradiction that $P_\mu(\phi)<0$ (note that $S_\mu(\phi)<S_\mu(Q_\mu)$ and $P_\mu(\phi)=0$ are not compatible). Since the function $\lambda \mapsto \Phi_\mu(\lambda)$ (see \eqref{defi-Phi}) is strictly decreasing on $(0,\infty)$ and $\lim_{\lambda \rightarrow \infty} \Phi_\mu(\lambda)=-\infty$, we have from the fact $\Phi_\mu(1) = I_\mu(\phi) >0$ that there exists $\lambda_\mu >1$ such that $\Phi_\mu(\lambda_\mu)=0$ or $I_\mu(\phi_{\lambda_\mu})=0$. It follows that
		\[
		\left\{
		\begin{array}{rcl}
		\partial_\lambda S_\mu(\phi_\lambda) >0 &\text{if} & 0<\lambda<\lambda_\mu, \\
		\partial_\lambda S_\mu(\phi_\lambda) <0 &\text{if} & \lambda_\mu<\lambda <\infty.
		\end{array}
		\right.
		\]
		Since $P_\mu(\phi)<0$, we have from the fact $\lim_{\lambda \rightarrow 0} P_\mu(\phi_\lambda) = \frac{1}{2} \|\phi\|^2_{L^2}>0$ that there exists $\tilde{\lambda}_\mu \in (0,1)$ such that $P_\mu(\phi_{\tilde{\lambda}_\mu})=0$, hence by \eqref{chara-P}, $S_\mu(\phi_{\tilde{\lambda}_\mu}) \geq S_\mu(Q_\mu)$. It however contradicts $S_\mu(\phi_{\tilde{\lambda}_\mu}) \leq S_\mu(\phi) < S_\mu(Q_\mu)$. Hence $P_\mu(\phi)>0$ and $\phi \in \mathcal{A}^+_\mu$. 
		
		We next prove $\mathcal{A}^+_\mu \subset \mathcal{K}^+_\mu$. It follows by using \eqref{union} and the fact
		\[
		\mathcal{A}^+_\mu \cap \mathcal{K}^-_\mu \cap \Sigma =\emptyset
		\]
		which follows from Theorem $\ref{theo-gwp}$, Lemma $\ref{lem-blowup-Sigma}$ and Lemma $\ref{lem-posi-ener}$.
		
		This shows that $\mathcal{A}^+_\mu \equiv \mathcal{K}^+_\mu$. The proof is complete.
	\end{proof}

	\noindent {\bf Proof of Theorem $\ref{theo-blowup}$.}
	It follows immediately from Lemmas $\ref{lem-blow-nega-Sig}$, $\ref{lem-blow-nega-radi}$, $\ref{lem-blowup-Sigma}$, $\ref{lem-blowup-rad}$ and Lemma $\ref{lem-equi-A-K}$.
	\hfill $\Box$
	
	We end this section with the following observation.
	\begin{lemma} \label{lem-example}
		Let $\mu \in \{0,1\}$. It holds that $E_\mu(Q_\mu)>0$. Moreover, there exists $u_0 \in \Sigma$ satisfying $\|\nabla u_0\|_{L^2}<1, E_\mu(u_0) > 0$ and $u_0 \in \mathcal{K}^-_\mu$. 
	\end{lemma}
	
	\begin{proof}
		Denote
		\[
		\psi_\lambda(x):= \lambda Q_\mu(\lambda x).
		\]
		By Lemma $\ref{lem-I}$,
		\[
		\partial_\lambda E_\mu(\psi_\lambda) = \partial_\lambda S_\mu(\psi_\lambda) = \frac{1}{\lambda} I_\mu(\psi_\lambda), \quad I_\mu(\psi_\lambda)= \lambda^2 \Psi_\mu(\lambda),
		\]
		where $\lambda \mapsto \Psi_\mu(\lambda)$ is strictly decreasing on $(0,\infty)$. Since $\Psi_\mu(1)=I_\mu(Q_\mu) =0$, we infer that $\lambda \mapsto E_\mu(\psi_\lambda)$ and $\lambda \mapsto S_\mu(\psi_\lambda)$ attain their maxima at $\lambda =1$. If $E_\mu(Q_\mu) \leq 0$, then for $\lambda$ close to 1 and $\lambda<1$, we have
		\[
		E_\mu(\psi_\lambda)<0, \quad I_\mu(\psi_\lambda)>0, \quad S_\mu(\psi_\lambda) < S_\mu(Q_\mu)
		\]
		which implies that
		\[
		\psi_\lambda \in \mathcal{K}^+_\mu, \quad E_\mu(\psi_\lambda)<0, \quad \psi_\lambda \in \Sigma.
		\]
		This however is a contradiction because Lemma $\ref{lem-equi-A-K}$, Theorem $\ref{theo-gwp}$ and Lemma $\ref{lem-blow-nega-Sig}$. This shows that $E_\mu(Q_\mu)>0$. Moreover, for $\lambda$ close to 1 and $\lambda>1$, we have
		\[
		E_\mu(\psi_\lambda) > 0, \quad I_\mu(\psi_\lambda) <0, \quad S_\mu(\psi_\lambda) <S_\mu(Q_\mu).
		\]
		Moreover, for $\lambda$ close to 1 and $\lambda>1$, we also have $\|\nabla \psi_\lambda\|_{L^2} =\lambda \|\nabla Q_\mu\|_{L^2}<1$ since $\|\nabla Q_\mu\|_{L^2}<1$. Thus, there exists an initial data satisfying the desired properties.
	\end{proof}

	\section{Energy scattering for radially symmetric initial data}
	\label{S4}
	\setcounter{equation}{0}
	
	In this section, we give the proof of Theorem $\ref{theo-scat}$. To this end, we prepare some lemmas. 
	
	\begin{lemma} \label{lem-open-set}
		We have that
		\[
		\overline{\Ac}^+_1 = \Ac^+_1 \cup \{0\}
		\]
		is an open set of $H^1$,
		where $\overline{\Ac}^+_1$ is as in \eqref{defi-overline-A+}.
	\end{lemma}
	
	\begin{proof}
		Let $\phi \in \overline{\Ac}^+_1$, i.e. $S_1(\phi) <S_1(Q_1)$ and $P_1(\phi) \geq 0$. If $P_1(\phi)>0$, then $\phi \in \Ac^+_1$. Otherwise, if $P_1(\phi)=0$, then, by \eqref{chara-P}, we must have $\phi=0$. This shows that $\overline{\Ac}^+_1 = \Ac^+_1 \cup \{0\}$. 
		
		To see that $\overline{\Ac}^+_1$ is an open set of $H^1$, it suffices to show that there exists $\epsilon_0>0$ sufficiently small such that if $\phi \in H^1 \backslash \{0\}$ satisfying $\|\phi\|_{H^1} <\epsilon_0$, then $\phi \in \Ac^+_1$. In fact, we have
		\[
		S_1(\phi) \leq \frac{1}{2} \|\phi\|^2_{H^1} \leq \frac{1}{2} \epsilon_0^2 < S_1(Q_1)
		\]
		provided that $\epsilon_0>0$ is sufficiently small. On the other hand, by Corollary \ref{coro-refi-MT} and the Gagliardo-Nirenberg inequality, we have
		\[
		\int_{\R^2} F_1(\phi) dx \leq \int_{\R^2} e^{4\pi |\phi|^2} |\phi|^6 dx \leq C \|\phi\|^6_{L^6} \leq C \|\nabla \phi\|^4_{L^2} \|\phi\|^2_{L^2} \leq C \epsilon_0^4 \|\phi\|^2_{L^2} \leq \frac{1}{4} \|\phi\|^2_{L^2}
		\]
		for $\epsilon_0>0$ sufficiently small. It follows that $P_1(\phi) \geq \frac{1}{4}\|\phi\|^2_{L^2} >0$, hence $\phi \in \Ac^+_1$. The proof is complete.
	\end{proof}

	\begin{lemma} \label{lem-coer}
		Let $\mu=1$ and $\phi \in \overline{\Ac}^+_1$. Then it holds that 
		\[
		I_1(\phi) \geq \min \left\{ 2(S_1(Q_1)-S_1(\phi)), \int k_1(4\pi |\phi|^2) dx \right\},
		\]
		where
		\begin{align} \label{defi-ks}
		k_1(s) = \frac{1}{4\pi} \left( \frac{1}{2}s^2 e^s - s e^s + e^s-1\right).
		\end{align}
		Note that $k_1(s) \geq \frac{1}{24\pi} s^3$ for all $s \geq 0$.
	\end{lemma}
	
	\begin{proof}
		It suffices to prove the above estimate for $\phi \in \Ac^+_1$ since it holds trivially for $\phi =0$. 
		If  
		\[
		4 \|\nabla \phi\|^2_2 - \int \left.\partial_\lambda \left( \lambda \overline{\phi} f_1(\lambda \phi) - 2 F_1(\lambda \phi) \right) \right|_{\lambda=1} dx \geq 0,
		\]
		then 
		\begin{align*}
		I_1(\phi) &= \|\nabla \phi\|^2_2 - \int \overline{\phi} f_1(\phi) - 2F_1(\phi) dx \\
		&\geq \frac{1}{4} \int \left.\partial_\lambda \left( \lambda \overline{\phi} f_1(\lambda \phi) - 2 F_1(\lambda \phi) \right) \right|_{\lambda=1} dx  - \int \overline{\phi} f_1(\phi) - 2F_1(\phi) dx \\
		& = \int k_1(4\pi |\phi|^2) dx,
		\end{align*}
		where $k_1(s)$ is given in \eqref{defi-ks}. Note that a direct computation shows that
		\[
		\left.  \partial_\lambda \left( \lambda \overline{\phi} f_1(\lambda \phi) - 2F_1(\lambda \phi)\right)\right|_{\lambda=1} = 8 \pi |\phi|^4 \left( e^{4\pi |\phi|^2} -1\right)
		\]
		and
		\[
		\overline{\phi} f_1(\phi) - 2F_1(\phi) = \frac{1}{4\pi} \left( 4\pi |\phi|^2 e^{4\pi |\phi|^2} - e^{4\pi |\phi|^2} + 1 - 8\pi^2 |\phi|^4 \right).
		\]
		It is not hard to check that $k_1(s) \geq \frac{1}{24\pi} s^3$ for all $s \geq 0$. 
		
		We now consider the case 
		\[
		4 \|\nabla \phi\|^2_2 - \int \left.\partial_\lambda \left( \lambda \overline{\phi} f_1(\lambda \phi) - 2 F_1(\lambda \phi) \right) \right|_{\lambda=1} dx < 0.
		\]
		As in the proof of Lemma $\ref{lem-est-I}$, if we set $g(\lambda):= S_1(\phi_\lambda)$, then a direct computation shows that
		\[
		g'(\lambda) = \lambda \|\nabla \phi\|^2_2 - \lambda^{-3} \int \lambda \overline{\phi} f_1(\lambda \phi) - 2F_1(\lambda \phi) dx = \frac{I_1(\phi_\lambda)}{\lambda}
		\]
		and
		\begin{align*}
		(\lambda g'(\lambda))' &=2 \lambda \|\nabla \phi\|^2_2 + 2 \lambda^{-3} \int \lambda \overline{\phi} f_1(\lambda \phi) - 2F_1(\lambda \phi) dx - \lambda^{-2} \int \partial_\lambda \left( \lambda \overline{\phi} f_1(\lambda \phi) -2 F_1(\lambda \phi) \right) dx \\
		&= -2 g'(\lambda) + \lambda \left( 4 \|\nabla \phi\|^2_2 - \lambda^{-3} \int \partial_\lambda \left( \lambda \overline{\phi} f_1(\lambda \phi) -2 F_1(\lambda \phi) \right) dx \right).
		\end{align*}
		It is easy to check that 
		\[
		\lambda \mapsto 4 \|\nabla \phi\|^2_2 - \lambda^{-3} \int \partial_\lambda \left( \lambda \overline{\phi} f_1(\lambda \phi) -2 F_1(\lambda \phi) \right) dx
		\]	
		is a decreasing function on $(0,\infty)$. Thus
		\[
		4 \|\nabla \phi\|^2_2 - \lambda^{-3} \int \partial_\lambda \left( \lambda \overline{\phi} f_1(\lambda \phi) -2 F_1(\lambda \phi) \right) dx \leq 4 \|\nabla \phi\|^2_2 -  \int \left.\partial_\lambda \left( \lambda \overline{\phi} f_1(\lambda \phi) -2 F_1(\lambda \phi) \right)\right|_{\lambda=1} dx <0 
		\]
		for all $\lambda \geq 1$. This shows that
		\[
		(\lambda g'(\lambda))' \leq -2g'(\lambda), \quad \forall \lambda \geq 1.
		\]
		Now since $\phi \in \mathcal{A}^+_1$, $I_1(\phi)>0$ and there thus exists $\lambda_1>1$ such that $I_1(\phi_{\lambda_1})=0$ and hence $S_1(\phi_{\lambda_1}) \geq S_1(Q_1)$. Integrating the above inequality over $(1, \lambda_1)$, we get
		\[
		I_1(\phi_{\lambda_1}) - I_1(\phi) \leq  -2 \left(S_1(\phi_{\lambda_1}) - S_1(\phi) \right) \text{ or } I_1(\phi) \geq 2 \left(S_1(\phi_{\lambda_1}) - S_1(\phi) \right) \geq 2 \left(S_1(Q_1) - S_1(\phi) \right).
		\]
		The proof is complete.
	\end{proof}
	
	
	\begin{lemma}
		Let $\mu=1$. Let $u_0 \in \mathcal{A}^+_1$ and $u_0$ be radially symmetric. Let $u$ be the corresponding global solution to \eqref{NLS}. Then there exists $C=C(u_0,Q_1)>0$ such that for any $R>0$ and any $t\in \R$,
		\begin{align} \label{viri-A+}
		\frac{d^2}{dt^2} V_{\varphi_R}(t) \geq 8 I_1(\chi_R u(t)) - CR^{-2} -C \left(e^{CR^{-1}}-1\right),
		\end{align}
		where $\chi_R(x) = \chi(x/R)$ with $\chi \in C^\infty_0(\R^2)$ satisfying $0\leq \chi \leq 1$, $\chi=1$ on $B(0,1/2)$ and $\chi=0$ on $\R^2 \backslash B(0,1)$. 
	\end{lemma}	
	
	\begin{proof}
		We have from \eqref{seco-viri} that
		\[
		\frac{d^2}{dt^2} V_{\varphi_R}(t) = -\int \Delta^2 \varphi_R |u(t)|^2 dx + 4\sum_{j,k} \int \partial^2_{jk} \varphi_R \text{Re}(\partial_j \overline{u}(t) \partial_k u(t)) dx - 2 \int \Delta \varphi_R H_1(u(t)) dx,
		\]
		where
		\[
		H_1(u):= \overline{u} f_1(u) - 2F_1(u).
		\]
		Since $\varphi_R(x) = |x|^2$ for $0\leq |x| \leq R$, 
		\begin{align*}
		\frac{d^2}{dt^2} V_{\varphi_R}(t) &= 8 \left(\int_{|x| \leq R} |\nabla u(t)|^2 dx - \int_{|x| \leq R} H_1(u(t)) dx \right) \\
		&\mathrel{\phantom{=}} - \int_{|x|>R} \Delta^2\varphi_R |u(t)|^2 dx + 4 \sum_{j,k} \int_{|x|>R} \partial^2_{jk} \varphi_R \text{Re}(\partial_j \overline{u}(t) \partial_k u(t)) dx \\
		&\mathrel{\phantom{= - \int_{|x|>R} \Delta^2\varphi_R |u(t)|^2 dx}} - 2 \int_{|x|>R} \Delta \varphi_R H_1(u(t)) dx.
		\end{align*}
		Since $\|\Delta^2 \varphi_R\|_{L^\infty} \lesssim R^{-2}$, the conservation of mass implies that
		\[
		\int_{|x|>R} \Delta^2 \varphi_R |u(t)|^2 dx \lesssim R^{-2}.
		\]
		Since $u$ is radially symmetric, we use the fact
		\[
		\partial^2_{jk} = \left(\frac{\delta_{jk}}{r} -\frac{x_j x_k}{r^3} \right) \partial_r + \frac{x_j x_k}{r^2} \partial^2_r
		\]
		to get
		\[
		\sum_{j,k} \partial^2_{jk} \varphi_R \partial_j \overline{u} \partial_k u = \varphi''_R |\partial_ru|^2 \geq 0.
		\]
		On the other hand, noting that $H_1(u) = \frac{1}{4\pi}h_1(4\pi |u|^2)$ (see \eqref{defi-H-mu}) with 
		\begin{align} \label{defi-h-1}
		h_1(s) := s(e^s-1-s) - \left( e^s-1-s-\frac{s^2}{2}\right) \leq s(e^s-1), \quad \forall s \geq 0,
		\end{align}
		we infer by using the radial Sobolev embedding 
		\[
		\|u(t)\|^2_{L^\infty(|x|>R)} \lesssim R^{-1} \|\nabla u(t)\|_{L^2} \|u(t)\|_{L^2} \leq R^{-1} \|\nabla Q_1\|_{L^2} \|u_0\|_{L^2}
		\]
		and estimating as in \eqref{est-non} that
		\begin{align*}
		\int_{|x|>R} \Delta \varphi_R H_1(u(t)) dx &\lesssim \int_{|x|>R} |u(t)|^2 \left( e^{4\pi |u(t)|^2} -1\right) dx \\
		&\lesssim e^{CR^{-1}} -1.
		\end{align*}
		This shows that
		\[
		\frac{d^2}{dt^2} V_{\varphi_R}(t) \geq 8 \left( \int_{|x| \leq R} |\nabla u(t)|^2 dx - \int_{|x| \leq R} H_1(u(t)) dx\right) - CR^{-2} - C\left( e^{CR^{-1}} -1\right)
		\]
		for some $C=C(u_0,Q_1)>0$. Now let $\chi_R$ be as in \eqref{viri-A+}. We see that
		\begin{align*}
		\int |\nabla (\chi_R u(t))|^2 dx &= \int \chi^2_R |\nabla u(t)|^2 dx + \int |\nabla \chi_R|^2 |u(t)|^2 dx + 2\text{Re} \int \chi_R \overline{u}(t) \nabla \chi_R \cdot \nabla u(t) dx \\
		&= \int \chi^2_R |\nabla u(t)|^2 dx - \int \chi_R \Delta(\chi_R) |u(t)|^2 dx \\
		&= \int_{|x| \leq R} |\nabla u(t)|^2 dx - \int_{R/2 \leq |x| \leq R} (1-\chi_R^2) |\nabla u(t)|^2 dx - \int \chi_R\Delta(\chi_R) |u(t)|^2 dx
		\end{align*}
		and
		\begin{align*}
		\int H_1(\chi_R u(t)) dx &= \int_{|x| \leq R} H_1(u(t)) dx - \int_{R/2 \leq |x| \leq R} H_1(u(t)) - H_1(\chi_R u(t)) dx.
		\end{align*}
		Note that
		\[
		\int \chi_R \Delta(\chi_R) |u(t)|^2 dx \lesssim R^{-2} 
		\]
		and since $h_1$ defined in \eqref{defi-h-1} is increasing on $[0,\infty)$,
		\[
		\left|\int_{R/2 \leq |x| \leq R} H_1(u(t)) - H_1(\chi_R u(t)) dx \right| \leq 2 \int_{R/2 \leq |x| \leq R} H_1(u(t)) dx \lesssim e^{CR^{-1}}-1.
		\]
		This implies that
		\begin{align*}
		\int_{|x| \leq R} |\nabla u(t)|^2 dx &- \int_{|x| \leq R} H_1(u(t)) dx \\
		&= \int |\nabla (\chi_R u(t))|^2 dx - \int H_1(\chi_R u(t)) dx + \int_{R/2 \leq |x| \leq R} (1-\chi^2_R) |\nabla u(t)|^2 dx \\
		&\mathrel{\phantom{=}} +\int \chi_R \Delta(\chi_R) |u(t)|^2 dx + \int_{R/2 \leq |x| \leq R} H_1(u(t)) - H_1(\chi_Ru(t)) dx \\
		&\geq I_1(\chi_R u(t)) - CR^{-2} - C\left( e^{CR^{-1}}-1\right)
		\end{align*}
		for some $C=C(u_0,Q_1)>0$. The proof is complete.
	\end{proof}
	
	\begin{lemma} \label{lem-est-S1-chi-R}
		Let $\mu=1$. Let $u_0 \in\mathcal{A}^+_1$ and $u_0$ be radially symmetric. Let $u$ be the corresponding global solution to \eqref{NLS}. Then there exists $R_0=R_0(u_0,Q_1)>0$ sufficiently large such that 
		\begin{align} \label{est-S1-chi-R}
		S_1(\chi_R u(t)) \leq S_1 (u_0) + CR^{-2} + C \left(e^{CR^{-1}} -1\right)
		\end{align}
		for some constant $C=C(u_0,Q_1)>0$. In particular, we have
		\begin{align} \label{S-chi-R}
		S_1(\chi_R u(t)) < S_1(Q_1)
		\end{align}
		for all $R\geq R_0$ and all $t\in \R$.
	\end{lemma}
	\begin{proof}
		Since $\mathcal{A}^+_1$ is invariant under the flow of \eqref{NLS}, we have that $u(t) \in \mathcal{A}^+_1$ for all $t\in \R$. We also have that
		\[
		S_1(\chi_R u(t)) = \frac{1}{2} \|\nabla (\chi_R u(t))\|^2_{L^2} + \frac{1}{2} \|\chi_R u(t)\|^2_{L^2} - \int F_1(\chi_R u(t)) dx,
		\]
		where
		\begin{align*}
		\|\chi_R u(t)\|^2_{L^2} &\leq \|u(t)\|^2_{L^2}, \\
		\|\nabla (\chi_R u(t))\|^2_{L^2} &= \int \chi^2_R |\nabla u(t)|^2 dx - \int \chi_R \Delta(\chi_R) |u(t)|^2 dx \\
		&\leq \int |\nabla u(t)|^2 dx  + O(R^{-2})
		\end{align*}
		and
		\begin{align*}
		\int F_1(\chi_R u(t)) dx &= \int F_1(u(t)) dx + \int_{|x| > R/2} F_1(\chi_Ru(t)) - F_1(u(t)) dx \\
		&= \int F_1(u(t)) + O\left( e^{CR^{-1}} -1\right).
		\end{align*}
		Here we have used the fact
		\[
		F_1(u) \lesssim \left(e^{4\pi |u|^2} -1 \right)|u|^2.
		\]
		Thus
		\[
		S_1(\chi_R u(t)) \leq S_1(u(t)) + CR^{-2} + C \left( e^{CR^{-1}}- 1\right)
		\]
		for some constant $C=C(u_0,Q_1)>0$. This proves \eqref{est-S1-chi-R} as $S_1(u(t)) = S_1(u_0)$ due to the conservation of mass and energy. Next, we write $S_1(u_0) =S_1(Q_1) - \rho$ for some $\rho=\rho(u_0,Q_1)>0$. By choosing $R_0=R_0(u_0,Q_1)>0$ sufficiently large so that for any $R \geq R_0$,
		\[
		C R^{-2} + C \left( e^{CR^{-1}}-1\right) \leq \rho/2,
		\]
		we obtain
		\[
		S_1(\chi_R u(t)) <S_1(Q_1)
		\]
		for all $R\geq R_0$ and all $t\in \R$.
	\end{proof}

	\begin{lemma} \label{lem-L6}
		Let $\mu=1$. Let $u_0 \in \mathcal{A}^+_1$ and $u_0$ be radially symmetric. Let $u$ be the corresponding global solution to \eqref{NLS}. Let $R_0 = R_0(u_0,Q_1)$ be as in Lemma \ref{lem-est-S1-chi-R}. Then we have that $\chi_R u(t) \in \overline{\Ac}^+_1$ for all $R\geq R_0$ and all $t\in \R$. Moreover, we have that for any time interval $I \subset \R$, 
		\begin{align} \label{large-I}
		\int_I \|u(t)\|^6_{L^6} dt \leq C(u_0,Q_1) |I|^{1/3}.
		\end{align}
	\end{lemma}
	\begin{proof}
		By \eqref{S-chi-R}, it remains to show that
		\begin{align}\label{est-L6-proof}
		P_1(\chi_Ru(t)) \geq 0
		\end{align}
		for all $R\geq R_0$ and all $t\in \R$. Indeed, suppose that there exist $R_1\geq R_0$ and $t_1 \in \R$ such that $P_1(\chi_{R_1} u(t_1))<0$. Note that $\lim_{R \rightarrow \infty} P_1(\chi_R u(t_1)) = P_1(u(t_1)) >0$ since $\Ac^+_1$ is invariant under the flow of \eqref{NLS}. Let $R_2>R_1$ be the smallest value such that $P_1(\chi_{R_2} u(t_1)) =0$, that is, $P_1(\chi_R u(t_1)) <0$ for all $R\in [R_1, R_2)$. Since $P_1(\chi_{R_2} u(t_1)) =0$ and $S_1(\chi_{R_2} u(t_1)) <S_1(Q_1)$, we have $\chi_{R_2} u(t_1) \in \overline{\Ac}^+_1$. Since $\overline{\Ac}^+_1$ is open in $H^1$, we get that $P_1(\chi_R u(t_1)) \geq 0$ for $R<R_2$ and $R$ close to $R_2$. This contradicts to the choice of $R_2$, and \eqref{est-L6-proof} is proved. 
		
		Now, we claim that there exists $C=C(u_0,Q_1)>0$ such that
		\begin{align} \label{claim}
		I_1(\chi_R u(t)) \geq C \|\chi_R u(t)\|^6_{L^6}
		\end{align}
		for all $R\geq R_0$ and all $t\in \R$. In fact, since $\chi_R u(t) \in \overline{\Ac}^+_1$ for all $R\geq R_0$ and all $t\in \R$, we have from Lemma \ref{lem-coer} that
		\[
		I_1(\chi_R u(t)) \geq \min \left\{ 2(S_1(Q_1)-S_1(\chi_R u(t))), \int k_1(4\pi |\chi_R u(t)|^2) dx \right\}
		\]
		for all $R\geq R_0$ and all $t\in \R$, where $k_1$ is as in \eqref{defi-ks}. Thanks to the fact that
		\[
		\int k_1(4\pi |\phi|^2) dx \geq \frac{8\pi^2}{3} \|\phi\|^6_{L^6},
		\]
		it remains to show
		\begin{align} \label{est-S1-Q1}
		S_1(Q_1) - S_1(\chi_R u(t)) \geq C(u_0,Q_1) \|\chi_R u(t)\|^6_{L^6}.
		\end{align}
		Indeed, by \eqref{est-S1-chi-R}, we have
		\[
		S_1(Q_1) - S_1(\chi_R u(t)) \geq S_1(Q_1) - S_1(u_0) - CR^{-2} - C \left(e^{CR^{-1}}-1\right) \geq \frac{1}{2} (S_1(Q_1) - S_1(u_0))
		\]
		provided that $R\geq R_0$ is taken sufficiently large. On the other hand, by Sobolev embedding, we obtain
		\[
		\|\chi_R u(t)\|_{L^6} \leq \|u(t)\|_{L^6} \leq \|u(t)\|_{H^1} \leq C(u_0,Q).
		\]
		This shows \eqref{est-S1-Q1}, and the claim is proved. 
		
		From \eqref{claim} and \eqref{viri-A+}, we get
		\[
		\|\chi_R u(t)\|^6_{L^6} \lesssim \frac{d^2}{dt^2} V_{\varphi_R}(t) + CR^{-2} + C \left( e^{CR^{-1}} -1 \right).
		\]
		Using the fact $\left| \frac{d}{dt} V_{\varphi_R}(t) \right| \lesssim R$, we see that for any $T>0$,
		\[
		\int_0^T \int |\chi_R u(t)|^6 dx dt \lesssim R + \left[CR^{-2} + C \left( e^{CR^{-1}} -1\right) \right] T
		\]
		which implies that
		\begin{align} \label{est-T}
		\int_0^T \int_{|x| \leq R/2} |u(t)|^6 dx dt \lesssim R +\left[CR^{-2} + C \left( e^{CR^{-1}} -1\right) \right] T.
		\end{align}
		It follows from \eqref{est-T} and the fact
		\[
		\int_{|x|>R/2} |u(t)|^6 dx \leq \|u(t)\|^4_{L^\infty(|x|>R/2)} \|u(t)\|^2_{L^2} \lesssim R^{-2} \|\nabla u(t)\|^2_{L^2} \|u(t)\|^4_{L^2} \lesssim CR^{-2}
		\]
		that
		\[
		\int_0^T \|u(t)\|^6_{L^6} dt \lesssim R + \left[CR^{-2} + C \left( e^{CR^{-1}} -1\right) \right] T.
		\]
		Note that for $R>0$ sufficiently large
		\[
		e^{CR^{-1}}-1 \lesssim CR^{-2}.
		\]
		It follows that for $T>0$ sufficiently large, we choose $R=T^{1/3}$ and get
		\begin{align*}
		\int_0^T \|u(t)\|^6_{L^6} dt \lesssim T^{1/3}.
		\end{align*}
		By the same argument, we also have that for any time interval $I$,
		\begin{align*} 
		\int_I \|u(t)\|^6_{L^6} dt \lesssim |I|^{1/3}.
		\end{align*}
		Indeed, for $|I|$ sufficiently large, it follows from the above argument. For $|I|$ small, it follows from the Sobolev embedding $\|u(t)\|_{L^6} \lesssim \|u(t)\|_{H^1} \leq C(u_0,Q_1)$.
	\end{proof}
	
	\begin{remark}
		It is not hard to see that Lemmas $\ref{lem-open-set}$--$\ref{lem-L6}$ still hold true with $\mu=0$.
	\end{remark}

	To show the scattering, we need the following scattering criteria.
	\begin{proposition} \label{prop-crite}
		Let $\mu=1$. Let $u_0 \in \mathcal{A}^+_1$ and $u$ be the corresponding global solution to \eqref{NLS}. Assume that 
		\begin{align} \label{glo-bound}
		\|u\|_{L^8(\R \times \R^2)} <\infty.
		\end{align}
		Then the solution scatters in $H^1$.
	\end{proposition}
	
	\begin{proof}
		We first notice that $u(t) \in \mathcal{A}^+_1$ for all $t\in \R$ and hence $\|\nabla u(t)\|_{L^2} < \|\nabla Q_1\|_2 <1$ for all $t\in \R$. This allows us to use the refined Moser-Trundinger inequality \eqref{MT-beta}. Let $I$ be a time interval. By Strichartz estimates and the Duhamel formula
		\[
		u(t) = e^{it\Delta} u_0 + i \int_0^t e^{i(t-s)\Delta} f_1(u(s)) dx,
		\]
		we have that
		\[
		\|u(t) - e^{it\Delta} u_0\|_{{\ST}(I)} \lesssim \|f_1(u)\|_{{\ST}^*(I)}.
		\]
		We refer to \eqref{defi-ST-norm} and \eqref{defi-ST-dual} for the definitions of ST- and ST*-norms. Using the fact $|f'_1(u)| \lesssim e^{4\pi|u|^2} |u|^4$, the H\"older's inequality implies that
		\[
		\|\scal{\nabla} f_1(u)\|_{L^{4/3}(I \times \R^2)} \lesssim \|\scal{\nabla} u\|_{L^4(I\times \R^2)} \left( \iint_{I \times \R^2} e^{8\pi |u|^2} |u|^8 dx dt \right)^{1/2}.
		\]
		By \eqref{MT-beta},
		\[
		\int_{\R^2} e^{8\pi |u(t)|^2} |u(t)|^8 dx \leq e^{4\pi\|u(t)\|^2_\infty} \int_{\R^2} e^{4\pi |u(t)|^2} |u(t)|^8 dx \lesssim_{Q_1} e^{4\pi \|u(t)\|^2_\infty} \|u(t)\|^8_{L^8}.
		\]
		
		In the case $\|u(t)\|_{L^\infty} \geq 1$, we use \eqref{log} with $\beta=1/2$ to get
		\[
		e^{4\pi \|u(t)\|^2_{L^\infty}} \lesssim \left(1+\frac{\|u(t)\|_{\mathcal{C}^{1/2}}}{\|u(t)\|_{H_\omega}} \right)^{4\pi \lambda \|u(t)\|^2_{H_\omega}}
		\]
		for some $\lambda >\frac{1}{\pi}$ and some $0<\omega<1$ to be chosen later. Since 
		\[
		\|u(t)\|^2_{H_\omega} = \|\nabla u(t)\|^2_{L^2} + \omega^2 \|u(t)\|^2_{L^2} < \|\nabla Q_1\|^2_{L^2} + \omega^2 \|u_0\|^2_{L^2} =: K^2(\omega), 
		\]
		we bound
		\[
		e^{4\pi \|u(t)\|^2_{L^\infty}} \lesssim  \left( 1+ \frac{\|u(t)\|_{\mathcal{C}^{1/2}}}{K(\omega)} \right)^{4\pi \lambda K^2(\omega)}.
		\]
		Here we have used the fact that the function 
		\[
		s \mapsto \left(1+\frac{1}{s}\right)^{s^2}
		\]
		is an increasing function on $(0,\infty)$. Since $K^2(\omega) \rightarrow \|\nabla Q_1\|^2_{L^2} <1$ as $\omega \rightarrow 0$, we can choose $\omega>0$ small enough depending on $u_0$ and $Q_1$ such that $K^2(\omega) <\frac{1}{2} \left(\|\nabla Q_1\|^2_{L^2}+1\right)$. We next choose $\lambda>\frac{1}{\pi}$ depending on $\omega$ so that $4\pi \lambda K^2(\omega) = 2 \left(\|\nabla Q_1\|^2_{L^2}+1\right)$. We thus get from \eqref{sobo-holder} that
		\[
		e^{4\pi \|u(t)\|^2_{L^\infty}} \lesssim_{u_0, Q_1} \left(1+ \|u(t)\|_{\mathcal{C}^{1/2}}\right)^m \lesssim_{u_0,Q_1} \|u(t)\|^m_{W^{1,4}},
		\]
		where $m:= 2 \left(\|\nabla Q_1\|^2_{L^2}+1\right) \in (2,4)$. This shows that
		\[
		\int_{\R^2} e^{8\pi |u(t)|^2} |u(t)|^8 dx \lesssim_{u_0,Q_1} \|u(t)\|^m_{W^{1,4}} \|u(t)\|^8_{L^8} \lesssim_{u_0,Q_1} \|u(t)\|^m_{W^{1,4}} \|u(t)\|^n_{L^8},
		\]
		where $n:= 2(4-m) \in (0,4)$. Here we have use the fact $\|u(t)\|_{L^8} \lesssim \|u(t)\|_{H^1} \lesssim C(u_0,Q_1)$. It follows that 
		\[
		\iint_{I \times \R^2} e^{8\pi |u|^2} |u|^8 dx dt \lesssim_{u_0,Q_1} \| \|u(t)\|^m_{W^{1,4}} \|_{L^{4/m}} \| \|u(t)\|^n_{L^8} \|_{L^{4/(4-m)}} = \|u\|_{L^4(I, W^{1,4})}^m \|u\|^n_{L^8(I\times \R^2)}.
		\]
		We thus get
		\begin{align*}
		\|\scal{\nabla} f(u)\|_{L^{4/3}(I \times \R^2)} \lesssim_{u_0,Q_1} \|u\|^{1+m/2}_{L^4(I, W^{1,4})} \|u\|^{n/2}_{L^8(I \times \R^2)}.
		\end{align*}
		
		In the case $\|u(t)\|_{L^\infty} \leq 1$, we simply bound
		\[
		\int_{\R^2} e^{8\pi |u(t)|^2} |u(t)|^8 dx \lesssim_{Q_1} \|u(t)\|^8_{L^8} \lesssim_{u_0,Q_1} \|u(t)\|^m_{W^{1,4}} \|u(t)\|^n_{L^8},
		\]
		where we have use $\|u(t)\|_{L^8}\lesssim \|u(t)\|_{W^{1,4}}$, $\|u(t)\|_{L^8} \lesssim \|u(t)\|_{H^1} \lesssim C(u_0,Q_1)$ and the fact $m+n = 8 -m <8$. By H\"older's inequality, we get
		\begin{align*} 
		\|\scal{\nabla} f(u)\|_{L^{4/3}(I\times \R^2)} \lesssim_{u_0,Q_1} \|u\|^{1+m/2}_{L^4(I,W^{1,4})} \|u\|^{n/2}_{L^8(I\times \R^2)}.
		\end{align*}
		Thus in both cases, we have proved that 
		\begin{align*} 
		\|\scal{\nabla} f(u)\|_{L^{4/3}(I\times \R^2)} \lesssim_{u_0,Q_1} \|u\|^{1+m/2}_{L^4(I,W^{1,4})} \|u\|^{n/2}_{L^8(I\times \R^2)}.
		\end{align*}
		Thus
		\[
		\|u(t)- e^{it\Delta} u_0\|_{{\ST}(I)} \lesssim_{u_0,Q_1} \|u\|^{1+m/2}_{L^4(I, W^{1,4})} \|u\|^{n/2}_{L^8(I\times \R^2)}
		\]
		which implies that
		\begin{align} \label{inter-I}
		\|u\|_{{\ST}(I)} \leq C(u_0,Q_1) + C(u_0,Q_1)  \|u\|^{1+m/2}_{L^4(I, W^{1,4})} \|u\|^{n/2}_{L^8(I\times \R^2)}.
		\end{align}
		Let $\vareps>0$ to be chosen shortly. By the assumption $\|u\|_{L^8(\R \times \R^2))} <\infty$, we split $\R$ into $J=J(\vareps)$ intervals $I_j$ such that
		\[
		\|u\|_{L^8(I_j \times \R^2)} <\vareps, \quad j=1, \cdots, J.
		\]
		Applying \eqref{inter-I} to $I_j$ with $j=1,\cdots, J$, we have that
		\[
		\|u\|_{{\ST}(I_j)} \leq C(u_0,Q_1) + C(u_0,Q_1) \vareps^{n/2} \|u\|^{1+m/2}_{{\ST}(I_j)}.
		\]
		By choosing $\vareps>0$, the continuity argument shows that
		\[
		\|u\|_{{\ST}(I_j)} \leq C(u_0,Q_1), \quad j=1, \cdots, J.
		\]
		Summing over all $j=1, \cdots, J$, we obtain $\|u\|_{{\ST}(\R)} \leq C(u_0,Q_1)<\infty$. This global bound implies the scattering.
	\end{proof}
	
		By the same argument as above with $m=2(\|\nabla Q_1\|^2_{L^2}+1)$ and $n=4-m$, we have the following scattering criteria in the case $\mu=0$.
	\begin{corollary}
		Let $\mu=0$. Let $u_0 \in \mathcal{A}^+_0$ and $u$ be the corresponding global solution to \eqref{NLS}. Assume that 
		\[
		\|u\|_{L^4(\R\times \R^2)} <\infty.
		\]
		Then the solution scatters in $H^1$.
	\end{corollary}

	\begin{proposition} \label{prop-scat}
		Let $\mu=1$. Let $u_0 \in \mathcal{A}^+_1$ and $u_0$ be radially symmetric. Let $u$ be the corresponding global solution to \eqref{NLS}. Then \eqref{glo-bound} holds.
	\end{proposition}
	
	\begin{proof}
		Let $\vareps>0$ be a small parameter to be chosen sufficiently small depending on $u_0$ and $Q_1$ below. By Sobolev embedding and Strichartz estimates,
		\[
		\|e^{it\Delta} u_0\|_{L^6(\R \times \R^2)} \lesssim \||\nabla|^{1/3}e^{it\Delta} \scal{\nabla} u_0\|_{L^6(\R, L^3)} \lesssim \|u_0\|_{H^1}.
		\]
		We may split $\R$ into $K=K(\vareps, u_0)$ intervals $I_k$ such that
		\begin{align} \label{bound-u0}
		\|e^{it\Delta} u_0\|_{L^6(I_k\times \R^2)} <\vareps
		\end{align}
		for all $k=1,\cdots,K$. Let $T=T(\vareps,u_0,Q_1)$ be a large parameter to be chosen later. We will prove that
		\begin{align} \label{est-l6}
		\|u\|_{L^6(I_k\times \R^2)} \lesssim T
		\end{align}
		for all $k=1,\cdots,K$. Summing over all intervals $I_k, k=1,\cdots,K$, we get
		\[
		\|u\|_{L^6(\R \times \R^2)} \lesssim T
		\]
		which implies the scattering. In fact, by the scattering criteria given in Proposition $\ref{prop-crite}$, it suffices to show
		\begin{align}\label{est-l8}
		\|u\|_{L^8(\R \times \R^2)} \leq C(\vareps,u_0,Q_1).
		\end{align}
		To see \eqref{est-l8}, we use \eqref{inho-stri-est} to have
		\begin{align*}
		\left\| \int_0^t e^{i(t-s)\Delta} f_1(u(s)) ds \right\|_{L^8(\R \times \R^2)} \lesssim \|f_1(u)\|_{L^{m'}(\R,L^{n'})},
		\end{align*}
		where $(m',n')$ is the dual pair of a Schr\"odinger acceptable pair $(m,n)$ satisfying 
		\[
		\frac{2}{m}+\frac{2}{n}=\frac{3}{2} \quad \text{or} \quad \frac{2}{m'}+\frac{2}{n'}=\frac{5}{2}.
		\]
		Using the fact $|f_1(u)| \lesssim e^{4\pi|u|^2} |u|^5$, we see that
		\begin{align*}
		\|f_1(u)\|_{L^{m'}(\R,L^{n'})} &\lesssim \|u\|_{L^\infty(\R,L^a)} \|e^{4\pi|u|^2}|u|^4\|_{L^b(\R,L^c)}
		\end{align*}
		for some $(a,b,c) \in [1,\infty]^3$ provided that 
		\[
		\frac{1}{m'}= \frac{1}{b}, \quad \frac{1}{n'} = \frac{1}{a}+\frac{1}{c}.
		\]
		Taking $c:=1+\nu$ for some $\nu>0$ small to be chosen shortly, we use \eqref{MT-nu} with $\|\nabla u(t)\|_{L^2} <\|\nabla Q_1\|_2<1$ for all $t\in \R$ to get
		\[
		\|e^{4\pi |u(t)|^2} |u(t)|^4\|^c_{L^c} = \int_{\R^2} e^{4\pi(1+\nu)|u(t)|^2} |u(t)|^{4(1+\nu)} dx \lesssim_{Q_1} \|u(t)\|^{4(1+\nu)}_{L^{4(1+\nu)}}
		\]
		provided that $0<\nu<\frac{1}{\|\nabla Q_1\|_{L^2}^2} -1$. This implies that
		\[
		\|e^{4\pi |u(t)|^2} |u(t)|^4\|_{L^c} \lesssim_{Q_1} \|u(t)\|^4_{L^{4c}}.
		\]
		Choosing $a=\frac{(1+\nu)(2+\nu)}{\nu}$, the Sobolev embedding implies
		\[
		\|u\|_{L^\infty(\R,L^a)} \lesssim \|u\|_{L^\infty(\R, H^1)}.
		\]
		It follows that
		\begin{align*}
		\|f_1(u)\|_{L^{m'}(\R,L^{n'})} &\lesssim_{Q_1} \|u\|_{L^\infty(\R,H^1)} \|u\|^4_{L^{4b}(\R,L^{4c})} \\
		&\lesssim_{Q_1} \|u\|_{L^\infty(\R,H^1)} \|u\|^{4\theta}_{L^6(\R\times \R^2)} \|u\|^{4(1-\theta)}_{L^\infty(\R, L^q)}
		\end{align*}
		provided that $\theta \in (0,1)$, $q\in(2,\infty)$ and
		\[
		\frac{1}{4b} =\frac{\theta}{6},\quad \frac{1}{4c} = \frac{\theta}{6} + \frac{1-\theta}{q}.
		\]
		We see that
		\[
		\frac{2}{a} + \frac{4\theta}{3} = \frac{5}{2}-\frac{2}{c},
		\]
		hence
		\[
		\theta = \frac{3}{4} \left(\frac{5}{2}-\frac{2}{c}-\frac{2}{a}\right) = \frac{3}{4} \left( \frac{5}{2}-\frac{2}{1+\nu}-\frac{2\nu}{(1+\nu)(2+\nu)}\right)=\frac{3(2+5\nu)}{8(2+\nu)}. 
		\]
		A direct computation shows
		\[
		m=\frac{4(2+\nu)}{6-\nu}, \quad n=\frac{2+\nu}{\nu}, \quad b=\frac{4(2+\nu)}{2+5\nu}, \quad q=\frac{2(10-7\nu)(1+\nu)}{6-3\nu-5\nu^2}.
		\]
		By taking $\nu>0$ sufficiently small, it is easy to check that $(m,n)$ is a Schr\"odinger acceptable pair and
		\begin{align*} 
		q\in (2,\infty), \quad \theta \in (0,1), \quad 4\theta >1.
		\end{align*}
		The Sobolev embedding then implies that
		\begin{align}
		\left\| \int_0^t e^{i(t-s)\Delta} f_1(u(s)) ds \right\|_{L^8(\R \times \R^2)} &\lesssim_{Q_1} \| u\|^{1+4(1-\theta)}_{L^\infty(\R,H^1)} \|u\|^{4\theta}_{L^6(\R\times \R^2)}  \label{inho-est-app} 	\\
		&\lesssim C(\vareps,u_0,Q_1). \nonumber
		\end{align}
		We thus get
		\[
		\|u\|_{L^8(\R\times \R^2)} \leq \|e^{it\Delta} u_0\|_{L^8(\R\times \R^2)} + C(\vareps,u_0,Q_1) \leq C(\vareps,u_0,Q_1)
		\]
		which proves \eqref{est-l8}. 
		
		It remains to show \eqref{est-l6}. By Sobolev embedding, we observe that
		\[
		\|u\|^6_{L^6(I\times \R^2)} \leq |I| \|u\|^6_{L^\infty(I, L^6)}  \lesssim |I| \|u\|^6_{L^\infty(I,H^1)} \lesssim_{u_0,Q_1} |I|
		\]
		for any interval $I \subset\R$. 
		It suffices to show \eqref{est-l6} with $|I_k| >2T$. Let us fix one such interval, say $I=(c,d)$ with $|I|>2T$. We will show that there exists $t_1 \in (c,c+T)$ such that
		\begin{align} \label{bound-t1}
		\left\| \int_0^{t_1} e^{i(t-s)\Delta} f_1(u(s)) ds \right\|_{L^6([t_1,+\infty) \times \R^2)} \leq C(u_0, Q_1) \vareps^{1/4}.
		\end{align}
		Assume \eqref{bound-t1} for the moment, let us prove \eqref{est-l6}. By the Duhamel formula
		\[
		e^{i(t-t_1)\Delta} u(t_1) = e^{it\Delta} u_0 + i \int_0^{t_1} e^{i(t-s)\Delta} f_1(u(s)) ds
		\]
		and \eqref{bound-u0}, we see that
		\[
		\|e^{i(t-s)\Delta} u(t_1)\|_{L^6([t_1,d]\times \R^2)} \leq C(u_0,Q_1) \vareps^{1/4}. 
		\]
		By the same argument as in the proof of \eqref{inho-est-app} with $c=1+\nu$, $a=\frac{(1+\nu)(2+\nu)}{\nu}$ and
		\begin{align} \label{defi-theta-mn}
		\theta=\frac{1+2\nu}{2-\nu}, \quad m=\frac{3(2+\nu)}{4-\nu}, \quad n=\frac{2+\nu}{\nu}, \quad b=\frac{3(2+\nu)}{2(1+2\nu)}, \quad q=\frac{12(1-\nu)(1+\nu)}{4-3\nu-4\nu^2},
		\end{align}
		we have for $\nu>0$ sufficiently small that $4\theta>1$ and
		\begin{align}
		\left\| \int_{t_1}^t e^{i(t-s)\Delta} f_1(u(s)) ds \right\|_{L^6([t_1,d] \times \R^2)} &\lesssim_{Q_1} \|u\|^{1+4(1-\theta)}_{L^\infty([t_1,d], H^1)} \|u\|^{4\theta}_{L^6([t_1,d]\times \R^2)} \label{inho-est-app-1}\\
		&\leq C(u_0,Q_1) \|u\|^{4\theta}_{L^6([t_1,d]\times \R^2)}. \nonumber
		\end{align}
		This together with
		\[
		u(t) = e^{i(t-t_1)\Delta} u(t_1) + i \int_{t_1}^t e^{i(t-s)\Delta} f_1(u(s)) ds
		\]
		yield
		\begin{align*}
		\|u\|_{L^6([t_1,d]\times \R^2)} &\leq \|e^{i(t-t_1)\Delta} u(t_1)\|_{L^6([t_1,d]\times \R^2)} + C(u_0,Q_1)\|u\|^{4\theta}_{L^6([t_1,d] \times \R^2)}  \\
		&\leq C(u_0,Q_1) \vareps^{1/4} + C(u_0,Q_1) \|u\|^{4\theta}_{L^6([t_1,d]\times \R^2)}.
		\end{align*}
		By the continuity argument and the fact $4\theta>1$, we get
		\[
		\|u\|_{L^6([t_1,d]\times \R^2)} \leq C(\vareps,u_0,Q_1).
		\]
		On the other hand, since $t_1-c <T$,
		\[
		\|u\|_{L^6([c,t_1]\times \R^2)} \lesssim |t_1-c|^{\frac{1}{6}} \lesssim T^{\frac{1}{6}}
		\]
		hence \eqref{est-l6} follows. 
		
		Let us prove \eqref{bound-t1}. By time-translation, we may assume that $c=0$. We first claim that there exists $t_0 \in [T/4,T/2]$ such that 
		\begin{align} \label{bound-vareps}
		\int_{t_0}^{t_0+\vareps T^{2/3}} \|u(s)\|^6_{L^6} ds \leq C(u_0,Q_1) \vareps.
		\end{align}
		Indeed, we cover the interval $J:=[T/4,T/2]$ by $N \sim \vareps^{-1} T^{1/3}$ intervals $J_k$ of length $\vareps T^{2/3}$ to have that
		\[
		N \min_{1\leq k\leq N} \int_{J_k} \|u(s)\|^6_{L^6} ds \leq \sum_{k=1}^N \int_{J_k} \|u(s)\|^6_{L^6} ds = \int_J \|u(s)\|^6_{L^6} ds \leq C(u_0,Q_1) T^{1/3}.
		\]
		This implies that there exists $k_0 \in \{1,\cdots,N\}$ such that
		\[
		\int_{J_{k_0}} \|u(s)\|^6_{L^6} ds \leq C(u_0,Q_1) \vareps
		\]
		which proves the claim. Set
		\[
		t_1:= t_0 + \vareps T^{2/3}.
		\]
		Since $t_0 <T/2$, by enlarging $T$ if necessary, we may assume that $t_1<T$. We will estimate the left hand side of \eqref{bound-t1} by considering separately $[0,t_0]$ and $[t_0,t_1]$. We first treat $[0,t_0]$. For $t>t_1$, we use the dispersive estimate and H\"older's inequality to get
		\begin{align*}
		\left\|\int_0^{t_0} e^{i(t-s)\Delta} f_1(u(s)) ds \right\|_{L^\infty} &\lesssim \int_0^{t_0} |t-s|^{-1} \|f_1(u(s))\|_{L^1} ds \\
		&\lesssim_{Q_1} \int_0^{t_0} |t-s|^{-1} \|u(s)\|^5_{L^5} ds \\
		&\lesssim_{Q_1} \int_0^{t_0} |t-s|^{-1} \|u(s)\|^{9/2}_{L^6} \|u(s)\|^{1/2}_{L^2} ds \\
		&\lesssim_{u_0,Q_1} \left(\int_0^{t_0} \|u(s)\|^6_{L^6} ds \right)^{3/4} \||t-s|^{-1}\|_{L^4_s([0,t_0])} \\
		&\lesssim_{u_0,Q_1} T^{1/4} |t-t_0|^{-3/4} \\
		&\lesssim_{u_0,Q_1} T^{1/4} |t_1-t_0|^{-3/4} \\
		&\lesssim_{u_0,Q_1} \left( \vareps T^{2/3} \right)^{-3/4}.
		\end{align*}
		This implies that
		\[
		\left\| \int_0^{t_0} e^{i(t-s)\Delta} f_1(u(s)) ds \right\|_{L^\infty([t_1,+\infty)\times \R^2)} \leq C(u_0,Q_1) \left( \vareps T^{2/3} \right)^{-3/4}.
		\]
		On the other hand, since
		\[
		i \int_0^{t_0} e^{i(t-s)\Delta} f_1(u(s)) ds  = e^{i(t-t_0)\Delta} u(t_0) - e^{it\Delta} u_0,
		\]
		Strichartz estimates imply that
		\[
		\left\| \int_0^{t_0} e^{i(t-s)\Delta} f_1(u(s)) ds \right\|_{L^4([t_1,+\infty) \times \R^2)} \leq C(u_0,Q_1).
		\]
		Interpolating between $L^\infty$ and $L^4$, we get
		\begin{align*}
		\Big\| \int_0^{t_0} &e^{i(t-s)\Delta} f_1(u(s)) ds \Big\|_{L^6([t_1,+\infty) \times \R^2)} \\
		&\leq \left\| \int_0^{t_0} e^{i(t-s)\Delta} f_1(u(s)) ds \right\|^{\frac{1}{3}}_{L^\infty([t_1,+\infty)\times \R^2)}     \left\| \int_0^{t_0} e^{i(t-s)\Delta} f_1(u(s)) ds \right\|^{\frac{2}{3}}_{L^4([t_1,+\infty) \times \R^2)} \\
		&\leq C(u_0,Q_1) \left(\vareps T^{2/3} \right)^{-\frac{1}{3}}.
		\end{align*}
		On $[t_0,t_1]$, we use \eqref{inho-est-app-1} and \eqref{bound-vareps} to have that
		\begin{align*}
		\left\| \int_{t_0}^{t_1} e^{i(t-s)\Delta} f_1(u(s)) ds \right\|_{L^6([t_1,\infty) \times \R^2)} 	&\lesssim_{Q_1} \|u\|^{1+4(1-\theta)}_{L^\infty([t_0,t_1],H^1)} \|u\|^{4\theta}_{L^6([t_0,t_1]\times \R^2)} \\
		&\leq C(u_0,Q_1) \vareps^{2\theta/3} \\
		&\leq C(u_0,Q_1) \vareps^{1/4},
		\end{align*}
		where $2\theta/3 >1/4$ with $\theta$ as in \eqref{defi-theta-mn}. Collecting the contributions of the above two intervals, we get
		\[
		\left\| \int_0^{t_1} e^{i(t-s)\Delta} f_1(u(s)) ds \right\|_{L^6([t_1,+\infty)\times \R^2)} \leq C(u_0,Q_1) \left[ \left(\vareps T^{2/3}\right)^{-1/3} + \vareps^{1/4}\right].
		\]
		By taking $T=\vareps^{-21/8}$, we prove \eqref{bound-t1}. The proof is complete.
	\end{proof}
	
		\begin{remark}
		The above argument does not work for $\mu=0$. The first difficulty is that an estimate similar to \eqref{inho-est-app-1}, namely
		\[
		\left\| \int_0^t e^{i(t-s)\Delta} f_0(u(s)) ds \right\|_{L^4(\R \times \R^2)} \lesssim_{Q_0} \|u\|^{1+4(1-\theta)}_{L^\infty(\R, H^1)} \|u\|^{4\theta}_{L^4(\R \times \R^2)}
		\]
		for some $\theta \in (0,1)$ satisfying $4\theta >1$, is not easy to obtain. More precisely, if we perform the same reasoning as above, we will get $\theta = \frac{2+3\nu}{2+\nu}$ which is strictly greater than 1. The second difficulty comes from the fact $(4,4)$ is a Schr\"odinger admissible pair which prevents the smallness of 
		\[
		\left\| \int_0^{t_0} e^{i(t-s)\Delta} f_0(u(s)) ds \right\|_{L^4([t_1,\infty) \times \R^2)}.
		\]
	\end{remark}

	\noindent {\it Proof of Theorem $\ref{theo-scat}$.}
	Theorem $\ref{theo-scat}$ follows immediately from Lemma $\ref{prop-crite}$ and Proposition $\ref{prop-scat}$.
	\hfill $\Box$

	\section*{Acknowledgement}
	This work was supported in part by the Labex CEMPI (ANR-11-LABX-0007-01). S. K benefited from the support of the project ODA (ANR-18-CE40-0020-02). V. D. D. would like to express his deep gratitude to his wife - Uyen Cong for her encouragement and support. M. M. is extremely thankful to his wife Souad for her support. The authors would like to thank Prof. Zihua Guo for the fruitful discussion which helps improve the manuscript. 

\end{document}